\documentclass[oneside]{amsart}
\usepackage{amssymb, a4wide, mathdots, url,hyperref,graphicx}

\usepackage{amsfonts}
\usepackage{amssymb}
\usepackage{amsfonts, amsmath, amsthm, amssymb}
\usepackage{graphicx}
\usepackage{listings}
\usepackage{xcolor}
\usepackage{thmtools}
\usepackage{mathrsfs}
\usepackage{hyperref}
\usepackage{marginnote}
\usepackage[all]{xy}
\usepackage{accents}
\usepackage{enumerate}

\renewcommand{\ss}{\operatorname{ss}}

\renewcommand\min{{\operatorname{min}}}

\theoremstyle{plain}
\newtheorem{thm}{Theorem}[section]

\newtheorem{prop}[thm]{Proposition}
\newtheorem{lem}[thm]{Lemma}
\newtheorem{cor}[thm]{Corollary}
\newtheorem{conj}[thm]{Conjecture}

\newtheorem{defn}[thm]{Definition}
\newtheorem{defn/thm}[thm]{Definition/Theorem}

\theoremstyle{remark}
\newtheorem{remark}{Remark}

\include{psfig}
\begin{document}
	
	\title{On irreducibility of six-dimensional compatible systems of $\mathbb{Q}$}
	
	\author{Boyi Dai}
	\address{Department of Mathematics, HKU, Pokfulam, Hong Kong}
	\email{Daiboy@connect.hku.hk}

	\begin{abstract}
		We study the irreducibility of 6-dimensional strictly compatible systems of $\mathbb{Q}$ with distinct Hodge-Tate weights. We prove that if one of the representations $\rho$ in such a system is irreducible and satisfies a self-dual condition $\rho^{\vee}\otimes\chi\cong\rho$ for some character $\chi$, then all but finitely many of them are irreducible.
	\end{abstract}

	\maketitle
	

	\section{Introduction}
	
	Given an elliptic curve $E$ over a number field $K$. As $\ell$ varies over rational primes, the rational $\ell$-adic Tate modules $\rho_{\ell}:=T_{\ell}(E)\otimes\mathbb{Q}_{\ell}$ and the $\ell$-torsion points $\overline{\rho}_{\ell}:=E[\ell]$ form classical examples of a (2-dimensional) compatible system and a mod $\ell$ compatible system, respectively. The following result shows that they share uniform irreducibility property. Here statement (i) can be regarded as a weak version of Serre's celebrated big image result in \cite{Se72}.
	\begin{thm}\label{Silicon}
		Let $E$ be an elliptic curve over a number field $K$, and consider the compatible system $\{\rho_{\ell}\}$ and the mod $\ell$ compatible system $\{\overline{\rho}_{\ell}\}$ of $\mathrm{Gal}_K$.
		\begin{enumerate}[(i)]
			\item If $E$ does not have complex multiplication over $\overline{K}$, then for every $\ell$, the $\ell$-adic representation $\rho_{\ell}$ is absolutely irreducible and Lie-irreducible. Moreover, the mod $\ell$ representation $\overline{\rho}_{\ell}$ is absolutely irreducible for sufficiently large $\ell$.
			\item If $E$ has complex multiplication over $K$, then after possibly enlarging the coefficients, the system $\{\rho_{\ell}\}$ can be written as a direct sum of two 1-dimensional compatible systems.
			\item If $E$ has potential complex multiplication but not over $K$, then $\{\rho_{\ell}\}$ is absolutely irreducible and is induced from a 1-dimensional compatible system of a quadratic extension of $K$, after necessarily enlarging the coefficients. Moreover, $\overline{\rho}_{\ell}$ is absolutely irreducible for sufficiently large $\ell$.
		\end{enumerate}
	\end{thm}
	
	It is generally believed that the results above should be extended to arbitrary semisimple compatible systems. More precisely, it is natural to conjecture the following. We call a compatible system is irreducible (resp. Lie-irreducible) if each representation is irreducible (resp. Lie-irreducible). Note also that in the elliptic curve cases above, the compatible systems are regular.
	\begin{conj}\label{Sulfur}
		Let $\mathcal{M}$ be an $n$-dimensional strictly compatible system of a number field $K$.
		\begin{enumerate}[(i)]
			\item $\mathcal{M}$ decomposes as a direct sum of irreducible strictly compatible systems.
			\item If $\mathcal{M}$ is irreducible, then its representations are residually irreducible for all but finitely many primes.
			\item If $\mathcal{M}$ is irreducible and regular, then it is induced from a Lie-irreducible strictly compatible system over a finite extension of $K$.
		\end{enumerate}
	\end{conj}
	
	For a survey of low-dimensional cases, see \cite[Section 1.2]{Hu23a}. When $K=\mathbb{Q}$ and under regularity condition, the case $n=4$ was studied in \cite{Hu23a}, and the case $n=5$ (includes partial irregular cases) was treated in \cite{DWW24}. The present paper addresses the case $n=6$. The main results are as follows.
	\begin{thm}\label{main}
		Let $\{\rho_{\lambda}:\mathrm{Gal}_{\mathbb{Q}}\to\text{GL}_6(\overline{E_{\lambda}})\}$ be a 6-dimensional $E$-rational regular strictly compatible system of $\mathbb{Q}$. Suppose that for some $\lambda_0$, the representation $\rho_{\lambda_0}$ is irreducible and satisfies $\rho_{\lambda_0}^{\vee}\otimes\chi\cong\rho_{\lambda_0}$ for some character $\chi$. Then $\rho_{\lambda}$ is irreducible for all but finitely many $\lambda$.
	\end{thm}
	
	\begin{cor}\label{main2}
		Let $\{\rho_{\lambda}:\mathrm{Gal}_{\mathbb{Q}}\to\text{GL}_6(\overline{E_{\lambda}})\}$ be a 6-dimensional $E$-rational pure, essentially self-dual, totally odd, regular strictly compatible system of $\mathbb{Q}$. Then there exists an integer $r\in\mathbb{N}$ such that:
		\begin{enumerate}[(i)]
			\item For all but finitely many $\lambda$, the represnetation decomposes as
			$$\rho_{\lambda}=\sigma_{\lambda,1}\oplus\sigma_{\lambda,2}\oplus\cdots\oplus\sigma_{\lambda,r}$$
			where each $\sigma_{\lambda,i}$ is irreducible.
			\item Each family $\{\sigma_{\lambda,i}\}$ extends to a strictly compatible system after possibly enlarging the coefficients.
		\end{enumerate}
	\end{cor}
	
	We arrange the article as follows. In \autoref{section2} we collect the necessary preliminaries for the proofs, including certain $\ell$-independence properties, big image results, weak abelian direct summands of $\ell$-adic representations, potential automorphy theorems which serve as the main tool in the proof, and certain results on semisimple reductions. \autoref{section3} contains the proofs of the main results. The proof of \autoref{main} is divided into two cases according to whether $\rho_{\lambda_0}$ is Lie-irreducible or not. They are treated separately in Section \ref{3.1} and \ref{3.2}. Finally, Section \ref{3.3} gives the proof of \autoref{main2}.
	
	\section{Preliminaries}\label{section2}
	
	\subsection{Compatible systems}
	\begin{defn}\label{Hydrogen}
		Let $K$ be a number field. An $n$-dimensional $E$-rational strictly compatible system of $\mathrm{Gal}_K$ is a datum
		$$\mathcal{M}=\left(E,S,\{P_v(T)\},\{\rho_{\lambda}\},\{\text{HT}_{\tau}\},\{\text{WD}_v\}\right)$$
		
		where:
		\begin{itemize}
			\item $E$ is a number field.
			\item $S$ is a finite set of primes of $K$, called the exceptional set.
			\item $P_v(T)\in E[T]$ is a degree $n$ monic polynomial for each prime $v\not\in S$ of $K$.
			\item $\rho_{\lambda}:\mathrm{Gal}_K\to\text{GL}_n(\overline{E_{\lambda}})$ is an $n$-dimensional continuous semisimple $\lambda$-adic representation.
			\item $\text{HT}_{\tau}$ is a multiset of $n$ integers for each embedding $\tau: K\hookrightarrow\overline{E}$.
			\item $\text{WD}_v$ is a semisimple Weil-Deligne representation of $K_v$ for each prime $v$ of $K$.
		\end{itemize}
		
		such that:
		\begin{enumerate}[(i)]
			\item Each $\rho_{\lambda}$ is geometric in the sense of Fontaine-Mazur with exceptional set $S$, that is
			\begin{itemize}
				\item $\rho_{\lambda}$ is unramified outside $S\cup S_{\lambda}$, where $S_{\lambda}$ consists of the primes of $K$ lying above the rational prime under $\lambda$;
				\item If $v\in S_{\lambda}$, then $\rho_{\lambda}|_{\mathrm{Gal}_{K_v}}$ is de Rham.
			\end{itemize}
			Moreover, $\rho_{\lambda}|_{\mathrm{Gal}_{K_v}}$ is crystalline when $v\in S_{\lambda}$ and $v\not\in S$.
			\item For each $v\not\in S\cup S_{\lambda}$, the characteristic polynomial of $\rho_{\lambda}(\text{Frob}_v)$ equals $P_v(T)$. 
			\item For each embedding $\tau:K\hookrightarrow\overline{E}$ and each $E$-embedding $\overline{E}\hookrightarrow\overline{E}_{\lambda}$, the Hodge-Tate weights of $\rho_{\lambda}$ is $\text{HT}_{\tau}$.
			\item For each $v\not\in S_{\lambda}$ and each isomorphism $\iota:\overline{E}_{\lambda}\cong\mathbb{C}$, the Frobenius semisimplified Weil-Deligne representation $\iota\text{WD}(\rho_{\lambda}|_{\mathrm{Gal}_{K_v}})^{\text{F-ss}}$ is isomorphic to $\text{WD}_v$.
		\end{enumerate}
	\end{defn}
	
	\begin{defn}
		An $E$-rational strictly compatible system $\mathcal{M}$ is called pure of weight $\omega$, if for each $v\not\in S$, for each root $\alpha$ of $P_v(T)$ in $\overline{E}$, and for each embedding $\iota:\overline{E}\hookrightarrow\mathbb{C}$ we have $|\iota(\alpha)|^2=(\#\kappa(v))^{\omega}$, where $\kappa(v)$ denotes the residue field of $K$ at $v$. The system $\mathcal{M}$ is called pure if it is pure of weight $\omega$ for some integer $\omega$.
	\end{defn}
	
	\begin{defn}
		An $E$-rational strictly compatible system $\mathcal{M}$ is called regular if for any embedding $\tau:K\hookrightarrow\overline{E}$, the $\tau$-Hodge-Tate weights are distinct.
	\end{defn}
	
	Under regularity condition, one can descend the coefficients of a strictly compatible system to $E_{\lambda}$ after enlarging $E$:
	\begin{lem}\cite[Lemma 5.3.1.(3)]{BLGGT14}\label{Chlorine}
		Let $\{\rho_{\lambda}\}$ be an $E$-rational strictly compatible system of $K$. Suppose it is regular, then after replacing $E$ by a finite extension, we may assume that for any open subgroup $H$ of $\mathrm{Gal}_K$, any place $\lambda$ and any $H$-subrepresentation $\sigma$ of $\rho_{\lambda}$, the representation $\sigma$ is defined over $E_{\lambda}$.
	\end{lem}
	
	Some of the results we use hold under conditions weaker than those in \autoref{Hydrogen}. We include this definition for accuracy.
	\begin{defn}
		Let $K$ be a number field. An $n$-dimensional $E$-rational Serre compatible system of $\mathrm{Gal}_K$ is a datum
		$$\mathcal{M}=\left(E,S,\{p_v(T)\},\{\rho_{\lambda}\}\right)$$
		where:
		\begin{itemize}
			\item $E$ is a number field.
			\item $S$ is a finite set of primes of $K$, called the exceptional set.
			\item $P_v(T)\in E[T]$ is a degree $n$ monic polynomial for each prime $v\not\in S$ of $K$.
			\item $\rho_{\lambda}:\mathrm{Gal}_K\to\text{GL}_n(E_{\lambda})$ is an $n$-dimensional continuous semisimple $\lambda$-adic representation.
		\end{itemize}
		such that:
		\begin{enumerate}[(i)]
			\item $\rho_{\lambda}$ is unramified outside $S\cup S_{\lambda}$, where $S_{\lambda}$ consists of the primes of $K$ lying above the rational prime under $\lambda$.
			\item For each $v\not\in S\cup S_{\lambda}$, the characteristic polynomial of $\rho_{\lambda}(\text{Frob}_v)$ equals $P_v(T)$. 
		\end{enumerate}
	\end{defn}
	
	\subsection{$\lambda$-independence}
	
	\begin{defn}
		Given a semisimple $\ell$-adic Galois representation $$\rho:\mathrm{Gal}_K\to\text{GL}_n(E_{\lambda})$$
		its algebraic monodromy group $\textbf{G}$ is defined as the Zariski closure of its image inside the algebraic group $\text{GL}_{n,E_{\lambda}}$.
	\end{defn}
	
	\begin{lem}\label{Helium}
		Let $\rho$ be a semisimple $\ell$-adic Galois representation, and let $\sigma$ be a subquotient of $\rho$. Then the algebraic monodromy groups of $\rho$ and of $\rho\oplus\sigma$ coincide.
	\end{lem}
	\begin{proof}
		The algebraic monodromy group of $\rho$ is the fundamental group of the Tannakian category generated by $\rho$ inside the category of $\ell$-adic Galois representations. The result follows from the fact that $\rho$ and $\rho\oplus\sigma$ generate the same Tannakian category.
	\end{proof}
	
	Due to the semisimplicity of $\rho$, the identity component $\textbf{G}^{\circ}$ of its algebraic monodromy group is a reductive group. We denote by $\textbf{G}^{\text{der}}$ the derived subgroup of $\textbf{G}^{\circ}$, which is semisimple. To describe $\lambda$-independent properties of compatible systems, the following notions are needed.
	
	\begin{defn}
		Let $F$ be a field and let $\textbf{G}\subseteq\mathrm{GL}_{n,F}$ be a reductive subgroup.
		\begin{enumerate}[(i)]
			\item Let $\textbf{T}$ be a maximal torus of $\textbf{G}\times\overline{F}$. The formal character of $\textbf{G}$ is the conjugacy class of $\textbf{T}$ in $\mathrm{GL}_{n,\overline{F}}$.
			\item Let $\textbf{T}'$ the maximal torus of $\textbf{G}^{\mathrm{der}}\times\overline{F}$. The formal bi-character of $\textbf{G}$ is the conjugacy class of the chain $\textbf{T}'\subseteq\textbf{T}$ in $\mathrm{GL}_{n,\overline{F}}$.
			\item Given two fields $F_1,F_2$ and two reductive groups $\textbf{G}_i\subseteq\mathrm{GL}_{n_i,F_i}, i=1,2$. We say they have the same formal character (resp. formal bi-character), if $n_1=n_2=n$ and there exists a split $\mathbb{Z}$-subtorus $\textbf{T}_{\mathbb{Z}}\subseteq\mathrm{GL}_{n,\mathbb{Z}}$ (resp. a chain of split $\mathbb{Z}$-subtori $\textbf{T}_{\mathbb{Z}}'\subseteq\textbf{T}_{\mathbb{Z}}\subseteq\mathrm{GL}_{n,\mathbb{Z}}$) such that $\textbf{T}_{\mathbb{Z}}\times\overline{F_i}$ (resp. $\textbf{T}_{\mathbb{Z}}'\times\overline{F}_i\subseteq\textbf{T}_{\mathbb{Z}}\times\overline{F}_i$) is contained in the formal character (resp. formal bi-character) of $\textbf{G}_i$ for each $i$. This defines an equivalence relation on formal characters (resp. formal bi-characters) of reductive subgroups of general linear groups over arbitrary fields.
			\item Let $\{F_i\}$ be a family of fields and let $\{\textbf{G}_i\subseteq\mathrm{GL}_{n,F_i}\}$ be a family of reductive groups. We say they have the same formal character (resp. same formal bi-character) if they belong to a single equivalence class defined in (iii). We say they have bounded formal characters (resp. bounded formal bi-characters) if they belong to finitely many such equivalence classes.
		\end{enumerate}
	\end{defn}
	
	We have the following standard $\lambda$-independence results on algebraic monodromy groups.
	\begin{thm}\cite{Se81}, \cite{Se84}, \cite[Theorem 3.19]{Hu13}\label{Lithium}.
		Given an $E$-rational Serre compatible system $\{\rho_{\lambda}:\mathrm{Gal}_K\to\text{GL}_n(E_{\lambda})\}$. Denote by $\textbf{G}_{\lambda}$ the algebraic monodromy group of $\rho_{\lambda}\otimes\overline{E_{\lambda}}$
		\begin{enumerate}[(i)]
			\item The component group $\pi_0(\textbf{G}_{\lambda})=\textbf{G}_{\lambda}/\textbf{G}_{\lambda}^{\circ}$ is independent of $\lambda$. In particular, the connectedness of $\textbf{G}_{\lambda}$ is independent of $\lambda$.
			\item The formal bi-character of the tautological representation $\textbf{G}_{\lambda}\hookrightarrow\text{GL}_{n,\overline{E_{\lambda}}}$ and hence the rank and semisimple rank of $\textbf{G}_{\lambda}$, are independent of $\lambda$.
		\end{enumerate}
	\end{thm}
	
	Denote by $(\overline{\rho}^{\text{ss}},\overline{V}^{\text{ss}})$ the semisimple reduction of a $\lambda$-adic Galois representation $(\rho,V)$. For investigating residual irreducibility, the following notion is needed.
	\begin{defn/thm}\cite[Theorem 3.1]{Hu23b}, \cite[Proposition 2.11]{Hu23a}.
		Given an $n$-dimensional regular $E$-rational strictly compatible system $\{(\rho_{\lambda},V_{\lambda})\}$ of a number field $K$. By \autoref{Chlorine}, after enlarging $E$, we may write the system as $\{\rho_{\lambda}:\mathrm{Gal}_K\to\mathrm{GL}_n(E_{\lambda})\}$. Write $d=[E:\mathbb{Q}]$. By restriction of scalars, we have an $nd$-dimensional $\mathbb{Q}$-rational compatible system:
		$$\left\{\rho_{\ell}:=\bigoplus_{\lambda|\ell}\rho_{\lambda}:\mathrm{Gal}_K\to\left(\mathrm{Res}_{E/\mathbb{Q}}\right)\left(\mathbb{Q}_{\ell}\right)\subseteq\mathrm{GL}_{nd}(\mathbb{Q}_{\ell})\right\}_{\ell}$$
		\begin{enumerate}[(i)]
			\item  There exists a finite Galois extension $L/K$ such that, for each sufficiently large $\ell$, up to isomorphism there exists a unique connected reductive group $$\underline{G}_{\ell}\subseteq\text{GL}_{nd,\mathbb{F}_{\ell}}$$
			called the algebraic envelope, satisfying
			\begin{itemize}
				\item $\overline{\rho_{\ell}}^{\text{ss}}(\mathrm{Gal}_L)$ is a subgroup of $\underline{G}_{\ell}(\mathbb{F}_{\ell})$ whose index is uniformly bounded as $\ell$ varies.
				\item $\underline{G}_{\ell}$ acts semisimply on the ambient space.
				\item The formal characters of the embeddings $\underline{G}_{\ell}\hookrightarrow\text{GL}_{nd,\mathbb{F}}$ are bounded as $\lambda$ varies.
			\end{itemize}
			\item For all but finitely many $\lambda$, let $\ell$ be the rational prime below $\lambda$ and let $(\sigma, W)$ be a subrepresentation of $\rho_{\lambda}\otimes\overline{\mathbb{Q}}_{\ell}$. Denote by $\underline{G}_W$ the image of $\underline{G}_{\ell}$ in $\text{GL}_{\overline{W}^{\text{ss}}}$, which is called the algebraic envelope of $W$.
		\end{enumerate}
	\end{defn/thm}
	
	\begin{thm}\cite[Theorem 3.12]{Hu23b}\label{Magnesium}
		Given an $n$-dimensional regular $E$-rational strictly compatible system $\{\rho_{\lambda}\}$ of a number field $K$. Except for finitely many $\lambda$, for any subrepresentation $(\sigma,W)$ of $\rho_{\lambda}$ we have:
		\begin{enumerate}[(i)]
			\item The algebraic envelope $\underline{G}_W$ and the algebraic monodromy $\textbf{G}_W$ of $\sigma$ have the same formal bi-character.
			\item There exists a finite Galois extension $L/K$, independent of $W$, such that the commutants of $\overline{\sigma}_{\lambda}^{\text{ss}}(\mathrm{Gal}_L)$ and $\underline{G}_{W}$ (resp. $[\overline{\sigma}_{\lambda}^{\text{ss}}(\mathrm{Gal}_L),\overline{\sigma}_{\lambda}^{\text{ss}}(\mathrm{Gal}_L)]$ and $\underline{G}_{W}^{\text{ss}}$) in $\text{End}(\overline{W})^{\text{ss}}$ are equal. In particular, $\overline{\sigma}_{\lambda}^{\text{ss}}(\mathrm{Gal}_L)$ (resp. $[\overline{\sigma}_{\lambda}^{\text{ss}}(\mathrm{Gal}_L),\overline{\sigma}_{\lambda}^{\text{ss}}(\mathrm{Gal}_L)]$) acts irreducibly on $\overline{W}^{\text{ss}}$ if and only if $\underline{G}_{W}$ (resp. $\underline{G}_{W}^{\text{der}}$) acts irreducibly on $\overline{W}^{\text{ss}}$.
			\item If $\textbf{G}_W$ is of type A and $\textbf{G}_W^{\circ}\to\mathrm{GL}_W$ is irreducible (in particular for Lie-irreducible representations of dimension $\le3$), then $\underline{G}_W$ and thus $\mathrm{Gal}_K$ (resp. $\mathrm{Gal}_{K^{\text{ab}}}$) act irreducibly on $\overline{W}^{\mathrm{ss}}$.
			\item If $\sigma$ is irreducible and of type A, then it is residually irreducible.
		\end{enumerate}
	\end{thm}
	
	We mention a recent result on $\ell$-independence of component groups of algebraic envelopes, which is an analogue of Serre's result \autoref{Lithium}(i). Given a smooth projective variety $X$ over some number field $K$, then the semisimplifications of the $\ell$-adic cohomology
	$$V_{\ell}:=\text{H}_{\text{\'et}}^i(X_{\overline{K}},\mathbb{Q}_{\ell})^{\text{ss}}$$
	(after extending scalars to $\overline{\mathbb{Q}}_{\ell}$) form a $\mathbb{Q}$-rational strictly compatible system of $K$. Such a system is called a \emph{compatible system arising from varieties}. For $\ell$ sufficiently large, define the full algebraic envelope of $V_{\ell}$ to be
	$$\widehat{\underline{G}_{\ell}}:=\text{Im}\overline{\rho_{\ell}}^{ss}\cdot\underline{G}_{\ell}$$
	\begin{thm}\cite[Theorem 1.3]{DH24}.
		Let $\{\rho_\ell:\mathrm{Gal}_K\to\text{GL}_n(\mathbb{Q}_\ell)\}_\ell$ be a semisimple compatible system arising from some variety, with algebraic monodromy groups $\{\textbf{G}_\ell\}_\ell$ 
		and full algebraic envelopes $\{\widehat{\underline{G}}_\ell\}_{\ell\gg0}$.
		Let $K^{\text{conn}}/K$ be the finite Galois extension corresponding to $\textbf{G}_\ell/\textbf{G}_\ell^\circ$ 
		which is independent of $\ell$.
		For all sufficiently large $\ell$, the finite Galois extension corresponding to the morphism 
		$$\mathrm{Gal}_K\stackrel{\bar\rho_\ell^{\ss}}{\longrightarrow}\widehat{\underline{G}}_\ell(\mathbb{F}_\ell)\to 
		\widehat{\underline{G}}_\ell(\mathbb{F}_\ell)/\underline{G}_\ell(\mathbb{F}_\ell)$$
		is $K^{\text{conn}}/K$. In particular, the component groups 
		$\pi_0(\textbf{G}_\ell)=\textbf{G}_\ell/\textbf{G}_\ell^\circ$ and $\pi_0(\widehat{\underline{G}}_\ell)=\widehat{\underline{G}}_\ell/\underline{G}_\ell$ 
		are naturally isomorphic for all $\ell\gg0$.
	\end{thm}

	\subsection{Lie-irreduciblility}
	
	\begin{defn}\label{Vanadium}
		A continuous Galois representation $\rho:\mathrm{Gal}_K\to\text{GL}_n(\overline{E_{\lambda}})$ is called Lie-irreducible if for any finite extension $L/K$, the restriction $\rho|_{\mathrm{Gal}_L}$ is irreducible.
	\end{defn}
	
	Given a finite extension of fields $L/K$ and a Galois representation $\sigma$ of $L$, we write the induced representation as
	$$\text{Ind}_L^K\sigma:=\text{Ind}_{\mathrm{Gal}_L}^{\mathrm{Gal}_K}\sigma$$
	\begin{prop}\cite[Proposition 3.4.1, Lemma 3.4.6]{Pa19}\label{Beryllium}.
		Given an irreducible Hodge-Tate $\lambda$-adic Galois representation $\rho:\mathrm{Gal}_K\to\mathrm{GL}_n(E_{\lambda})$ of some number field $K$. If for each embedding $\tau:K\hookrightarrow\overline{E}$ the $\tau$-Hodge-Tate weights are distinct, then either $\rho$ is Lie-irreducible, or is induced from a Lie-irreducible representation $\sigma$ of some finite extension $L/K$:
		$$\rho=\mathrm{Ind}_L^K\sigma$$
	\end{prop}
	
	Hence a regular Hodge-Tate semisimple Galois representation $\rho$ can be written as
	\begin{equation}\label{4}
	\rho=\oplus_i\mathrm{Ind}_{K_i}^K\sigma_i
	\end{equation}
	where each $\sigma_i$ is a Lie-irreducible representation of $K_i$. Denote by $\text{Spl}(\rho)$ the set of primes of $K$ that have a split factor in at least one of the extensions $K_i/K$. Then the usual formula for the trace of an induced representation shows that the set of primes of $K$ on which the trace of $\rho$ is nonzero is contained in $\text{Spl}(\rho)$. Moreover, \cite[Proposition 3.4.9.(1)]{Pa19} shows the Dirichlet density of these two sets coincide. Hence from the compatibility of the system we obtain:
	\begin{prop}\label{Boron}
		Given a regular Hodge-Tate $E$-rational Serre compatible system $\{\rho_{\lambda}\}$ of some number field $K$. For each $\lambda$, consider the decomposition of $\rho_{\lambda}$ guarateed by \autoref{Beryllium}:
		$$\rho_{\lambda}=\oplus_i\mathrm{Ind}_{K_{\lambda,i}}^K\sigma_{\lambda,i}$$
		with each $\sigma_i$ Lie-irreducible of $K_{\lambda,i}$. Denote by $d(\rho_{\lambda})$ the Dirichlet density of primes in $K$ that have a split factor in at least one of the extensions $K_{\lambda,i}/K$. Then this density is independent of $\lambda$.
	\end{prop}
	
	\begin{cor}\label{Argon}
		In particular, under the above conditions, if some $\rho_{\lambda_0}$ is induced from a representation of a nontrivial extension of $K$, then no irreducible constituent of any $\rho_{\lambda}$ can be Lie-irreducible.
	\end{cor}
	
	\subsection{Essential self-duality and oddness}
	
	\begin{defn}
		Let $K$ be a totally real field. Let $E$ be a number field and let $\lambda$ be a prime of $E$. A $\lambda$-adic Galois representation $\rho:\text{Gal}_K\to\text{GL}_n(\overline{E}_{\lambda})$ is called essentially self-dual, if it either factors through $\text{GSp}_n(\overline{E}_{\lambda})$ or $\text{GO}_n(\overline{E}_{\lambda})$. In particular there exists some continuous character $\chi:\text{Gal}_K\to\overline{E}_{\lambda}^*$, called a similitude character, such that $\rho\cong\rho^{\vee}\otimes\chi$. 
	\end{defn}
	
	Note that for an essentially self-dual representation, the similitude character may not be unique.
	
	\begin{defn}
	A Galois character $\chi$ of a totally real field is called totally odd (resp.  totally even), if for any complex conjugation $c$ one has $\chi(c)=-1$ (resp. $\chi(c)=1$).
	\end{defn}
	
	\begin{lem}\label{Manganese}
		Given an $E$-rational Serre compatible system $\{\rho_{\lambda}\}$ and extend the scalars of each $\rho_{\lambda}$ to $\overline{E}_{\lambda}$. If for some $\lambda_0$ we have $\rho_{\lambda_0}\cong\rho_{\lambda_0}^{\vee}\otimes\chi_{\lambda_0}$, then for each $\lambda$ there exists a character $\chi_{\lambda}$ such that $\rho_{\lambda}\cong\rho_{\lambda}^{\vee}\otimes\chi_{\lambda}$. Moreover if the similitude character $\chi_{\lambda_0}$ is totally odd (resp. totally even), then one may choose each similitude character $\chi_{\lambda}$ to be totally odd (resp. totally even).
	\end{lem}
	\begin{proof}
		We have $\rho_{\lambda_0}\otimes\rho_{\lambda_0}^{\vee}\otimes\chi_{\lambda}\cong\rho_{\lambda_0}\otimes\rho_{\lambda_0}$. As $\rho_{\lambda_0}\otimes\rho_{\lambda_0}^{\vee}$ contains a trivial subrepresentation, $\chi_{\lambda_0}$ is a subrepresentation of $\rho_{\lambda_0}\otimes\rho_{\lambda_0}$. By \cite[Theorem 1.1]{BH25}, this character $\chi_{\lambda_0}$ is locally algebraic. Hence after possibly enlarging $E$, it extends to a compatible system $\{\chi_{\lambda}\}$. Then by compatibility, we have $\rho_{\lambda}\cong\rho_{\lambda}^{\vee}\otimes\chi_{\lambda}$ for each $\lambda$. Finally, as $\chi_{\lambda}$ can be written as a product of a finite image character and some power of cyclotomic character, $\chi_{\lambda}(c)$ is independent of $\lambda$ for a fixed complex conjugation $c$.
	\end{proof}
	
	The notion of total oddness is extended to higher-dimensional representations as follows.
	\begin{defn}
		Let $K$ be a totally real field and let $E$ be a number field.
		\begin{enumerate}[(i)]
			\item A $\lambda$-adic Galois representation $\rho:\mathrm{Gal}_K\to\mathrm{GL}_n(\overline{E}_{\lambda})$ is called essentially self-dual and totally odd, if it either factors through $\text{GSp}_n$ with a totally odd similitude character or factors through $\text{GO}_n$ with a totally even similitude character.
			\item An $E$-rational strictly compatible system $\{\rho_{\lambda}\}$ of $K$ is called essentially self-dual and totally odd, if each $\rho_{\lambda}$ is essentially self-dual and totally odd.
		\end{enumerate}
	\end{defn}
	\subsection{Weak abelian direct summands}\
	
	In what follows we assume that $K$ is a number field.
	\begin{defn}
		Let $\rho:\text{Gal}_K\to\text{GL}_n(F)$ and $\psi:\text{Gal}_K\to\text{GL}_m(F)$ be two semisimple $\ell$-adic representations that are unramified almost everywhere. Denote by $S_{\rho}$ and $S_{\psi}\subseteq\Sigma_K$ the sets of ramified places of $\rho$ and $\psi$, respectively. We say that $\psi$ is a weak direct summand of $\rho$ if the set
		$$S_{\psi|\rho}:=\{v\in\Sigma_K\backslash(S_{\rho}\cup S_{\psi}):\det(\psi(\mathrm{Frob}_v)-T\cdot\mathrm{Id})\text{\ divides\ }\det(\rho(\mathrm{Frob}_v)-T\cdot\mathrm{Id})\}$$
		has Dirichlet density one. If $\psi$ is abelian and is a weak direct summand of $\rho$, we say that $\psi$ is a weak abelian direct summand of $\rho$.
	\end{defn}
	
	Let $\textbf{G}\subseteq\text{GL}_{n,F}$ be a closed subgroup that acts irreducibly on $F^n$. Suppse that $\textbf{G}/Z(\textbf{G})$ is connected, where $Z(\textbf{G})\subseteq\textbf{G}$ is the center, and suppose that the multiplicity $n_0$ of the weight zero in the formal character of $\textbf{G}^{\text{der}}$ on $F^n$ is nonzero. By \cite[Lemma 2.6]{BH25} we have $\textbf{G}^{\text{der}}\cap\mathbb{G}_m$ is trivial and a direct product
	$$\textbf{G}=\textbf{G}^{\text{der}}\times\left(\textbf{G}\cap\mathbb{G}_m\right)$$
	We denote by $v_{\textbf{G}}:\text{Gal}_K\to\textbf{G}(F)\to\left(\textbf{G}\cap\mathbb{G}_m\right)(F)\subseteq F^{\times}$ the Galois character given by projection onto the second factor.
	\begin{prop}\cite[Proposition 2.9]{BH25}\label{Copper}
		Let $\rho:\mathrm{Gal}_K\to\mathrm{GL}_n(F)$ be a semisimple $\ell$-adic representation unramified outside a finite subset $S\subseteq\Sigma_K$ that can be written as $\rho=\oplus_{j\in J}\rho_j$ with absolutely irreducible representations $\rho_j:\mathrm{Gal}_K\to\mathrm{GL}_{r_j}(F)$. Suppose that for each $j\in J$ the group $\textbf{G}_{\rho_j}/Z(\textbf{G}_{\rho_j})$ is connected, where $\textbf{G}_{\rho_j}$ is the algebraic monodromy group of $\rho_k$. Denote by $n_{j,0}$ the multiplicity of the weight zero in the formal character of $\textbf{G}_{\rho_j}^{\mathrm{der}}$ on $F^{r_j}$, set $J_0=\{j\in J:n_{j,0}>0\}$ and for $j\in J_0$ define $\xi_j:=v_{\textbf{G}_{\rho_j}}$. Suppose that for all $j,j'\in J_0$ the character $\xi_{j'}\xi_j^{-1}$ is either trivial or has infinite order. Define a partition $J_0=\coprod_{a\in A}J_a$ by requiring $j,j'$ to lie in the same class if and only if $\xi_j=\xi_{j'}$, and write $\xi_a$ for $\xi_j$ if $j\in J_a$. Then $$\rho^{\mathrm{wab}}=\oplus_{a\in A}\xi_{a}^{\sum_{j\in J_a}n_{j,0}}$$ 
		is a degree $n_0=\sum_{j\in J}n_{j,0}$ weak abelian direct summand of $\rho$ such that any weak abelian direct summand $\psi$ of $\rho$ is a subrepresentation of $\rho^{\mathrm{wab}}$. We call $\rho^{\mathrm{wab}}$ the weak abelian part of $\rho$.
	\end{prop}
	
	\subsection{Automorphic Galois representations}\ 
	
	A standard method for showing that a Galois representation fits into a strictly compatible system is to show that it is automorphic.
	\begin{thm}\cite[Theorem C]{BLGGT14}\label{Carbon}.
		Suppose $K$ is a totally real field. Let $n$ be an integer and $\ell\ge 2(n+1)$ be a prime. Let
		$$\rho:\mathrm{Gal}_K\to\text{GL}_n(\overline{\mathbb{Q}_{\ell}})$$
		be a continuous representation. Suppose that the following conditions are satisfies.
		\begin{enumerate}[(i)]
			\item (Unramified almost everywhere) $\rho$ is unramified at all but finitely many primes.
			\item (Odd essential self-duality) Either $\rho$ maps to $\text{GSp}_n$ with totally odd similitude character or it maps to $\text{GO}_n$ with totally even similitude character.
			\item (Potential diagonalizability and regularity) $\rho$ is potentially diagonalizable (and hence potentially crystalline) at each prime $v$ of $K$ above $\ell$ and for each $\tau:K\hookrightarrow\overline{\mathbb{Q}_{\ell}}$ it has $n$ distinct $\tau$-Hodge-Tate weights.
			\item (Irreducibility) $\rho|_{\mathrm{Gal}_{K(\zeta_{\ell})}}$ is residually irreducible.
		\end{enumerate}
		Then we can find a finite Galois totally real extension $K'/K$ such that $\rho|_{\mathrm{Gal}_{K'}}$ is automorphic. Moreover $\rho$ is part of a strictly pure compatible system of $K$.
	\end{thm}
	
	Condition (iii) admits the following criterion.
	\begin{lem}\cite[Lemma 1.4.3(2)]{BLGGT14}\label{Zinc}
		Let $K/\mathbb{Q}_{\ell}$ be a finite unramified extension. If $\rho:\mathrm{Gal}_K\to\mathrm{GL}_n(\overline{\mathbb{Q}}_{\ell})$ is crystalline and for each embedding $\tau:K\hookrightarrow\overline{\mathbb{Q}}$ the Hodge-Tate weights $\mathrm{HT}_{\tau}(\rho)\subseteq[a_{\tau},a_{\tau}+\ell-2]$ for some integer $a_{\tau}$, then $\rho$ is potentially diagonalizable.
	\end{lem}
	
	The next result shows that, under mild technical conditions, 2-dimensional representations of totally real fields are automatically totally odd for sufficiently large $\lambda$.
	\begin{prop}\cite[Proposition 2.5]{CG13}\label{Nitrogen}.
		Let $K$ be a totally real field and $\ell>7$ a prime. Let
		$$\rho:\mathrm{Gal}_K\to\mathrm{GL}_2(\overline{\mathbb{Q}_{\ell}})$$
		be a continuous representation satisfying:
		\begin{enumerate}[(i)]
			\item $\rho$ is unramified outside a finite set of primes;
			\item $\text{Sym}^2\overline{\rho}|_{\mathrm{Gal}_{K(\zeta_{\ell})}}$ is irreducible;
			\item $\ell$ is unramified in $K$;
			\item For each place $v\mid\ell$ of $K$, the restriction $\rho|_{\mathrm{Gal}_{K_v}}$ is crystalline. Moreover, for each embedding $\tau:K_v\hookrightarrow\overline{\mathbb{Q}_{\ell}}$, the $\tau$-Hodge-Tate weights of $\rho|_{\mathrm{Gal}_{K_v}}$ are two distinct integers whose difference is less than $(\ell-1)/2$.
		\end{enumerate}
		Then the pair $(\rho,\det\rho)$ is essentially self-dual and totally odd.
	\end{prop}
	
	Under the oddness condition, the work of many people \cite{SW01,HT15,Kis09,KW09,Pan22,Zh25} leads to the following modularity theorem.
	\begin{thm}\cite[Theorem 1.0.2]{Zh25}\label{Nickel}
		Let $p$ be an odd prime number and $\rho:\mathrm{Gal}_{\mathbb{Q}}\to\text{GL}_2(\overline{\mathbb{Q}}_p)$ be a continuous, irreducible representation such that
		\begin{itemize}
			\item $\rho$ is only ramified at finitely many places,
			\item $\rho_{G_{\mathbb{Q}_p}}$ is de Rham of distinct Hodge-Tate weights,
			\item $\rho$ is odd,
		\end{itemize}
		Then $\rho$ arises from a cuspidal eigenform up to twist.
	\end{thm}
	
	In particular, the above results \autoref{Carbon} and \autoref{Nitrogen}, together with the big image result \autoref{Magnesium}, imply that certain low-dimensional subrepresentations of strictly compatible systems fit into strictly compatible systems.
	\begin{prop}\cite[Proposition 2.12]{Hu23a}\label{Oxygen}
		Given an $E$-rational strictly compatible system $\{\rho_{\lambda}\}$ of some totally real field. Then for all but finitely many $\lambda$
		\begin{enumerate}[(i)]
			\item If $\sigma$ is a 2-dimensional irreducible regular subrepresentation of $\rho_{\lambda}$, then $\sigma$ extends to a 2-dimensional regular irreducible strictly compatible system.
			\item If $\sigma$ is a 3-dimensional irreducible regular essentially self-dual subrepresentation of $\rho_{\lambda}$, then $\sigma$ extends to a 3-dimensional regular irreducible strictly compatible system.
		\end{enumerate}
	\end{prop}
	
	An immediate consequence is the following.
	\begin{cor}\label{Fluorine}
		Given an $E$-rational strictly compatible system $\{\rho_{\lambda}\}$ of some totally real field. Suppose that there exists infinitely many $\lambda$ such that the irreducible decomposition of $\rho_{\lambda}$
		$$\rho_{\lambda}=\oplus_iW_i$$
		consists only of regular irreducible conponents of dimensions 1 or 2. Then the compatible system $\{\rho_{\lambda}\}$ decomposes as a direct sum of 1- and 2-dimensional irreducible strictly compatible systems.
	\end{cor}
	
	\subsection{Semisimple reduction}\ 
	
	We extract the following result from the proof of \cite[Theorem 1.4]{Hu23b}.
	\begin{prop}\label{Sodium}
		Let $\{\rho_{\lambda}:\mathrm{Gal}_{\mathbb{Q}}\to\text{GL}_n(\overline{E_{\lambda}})\}$ be an $E$-rational strictly compatible system of $\mathbb{Q}$. Consider the associated modulo $\lambda$ compatible system $\{\overline{\rho}_{\lambda}^{\text{ss}}\}$ obtained by taking semisimple reductions. Suppose that for infinitely many $\lambda$ there exists a 2-dimensional odd irreducible subrepresentation
		$$\sigma_{\lambda}\subseteq\overline{\rho}_{\lambda}^{\text{ss}}$$
		Then, after possibly twisting each such $\sigma_{\lambda}$ by a fixed power of the cyclotomic character, infinitely many of these $\sigma_{\lambda}$ are attached to a single cuspidal eigenform (in the sense of Serre's modularity conjecture). In particular, there exists a 2-dimensional irreducible compatible system $\{\widetilde{\sigma}_{\lambda}\}$ of $\mathbb{Q}$ such that for infinitely many $\lambda$, the semisimple reduction of $\widetilde{\sigma}_{\lambda}$ equals $\sigma_{\lambda}$.
	\end{prop}
	\begin{proof}
		
		Fix an integer $m$. Write $\ell=\ell(\lambda)$ for the rational prime below $\lambda$, and let $\varepsilon_{\ell}$ be the $\ell$-adic cyclotomic character. By Serre's modularity conjecture, each 2-dimensional irreducible and odd $\sigma_{\lambda}\otimes\varepsilon^m_{\lambda}$ arises from a Hecke eigenform $f_{\lambda}$ of minimal weight $k_{\lambda}$ and minimal level $N_{\lambda}$. We will show that for a suitable choice of $m$, the weights $k_{\lambda}$ and the levels $N_{\lambda}$ are bounded for infinitely many $\lambda$. Then the result follows from the fact that there can only be finitely many eigenforms with fixed weight and level.
		
		After twisting the system $\{\rho_{\lambda}\}$ by a suitable power of cyclotomic characters, we may assume that the Hodge-Tate weights of $\{\rho_{\lambda}\}$ lie in $[0,C]$ for some $C>0$. We first bound the weights $k_{\lambda}$. Let $I_{\ell}\subseteq\text{Gal}_{\mathbb{Q}_{\ell}}$ be the inertia subgroup. The semisimplification $\left(\sigma_{\lambda}|_{I_{\ell}}\right)^{\text{ss}}$ factors through the tame inertia $I^t_{\ell}$, yielding two tame inertia characters $\gamma_{\lambda}$ and $\gamma_{\lambda}'$. Since the representation here extends to the full decomposition group $D_{\ell}$, those two characters are stable under the Frobenius action $x\mapsto x^{\ell}$. By \cite[Section 2]{Da95} there are two cases:
		\begin{enumerate}
			\item[(1)] $\gamma_{\lambda}^{\ell}=\gamma_{\lambda}'$ and $(\gamma_{\lambda}')^{\ell}=\gamma_{\lambda}$. Then
			$$\gamma_{\lambda}=\overline{\theta}_2^{a_{\lambda}+\ell b_{\lambda}}$$
			where $\overline{\theta}_2$ is the fundamental tame inertia character of level 2 and we can normalize the exponents to be $0\le a_{\lambda},b_{\lambda}\le\ell-1$.
			\item[(2)] $\gamma_{\lambda}^{\ell}=\gamma_{\lambda}$ and $(\gamma_{\lambda}')^{\ell}=\gamma_{\lambda}'$. Then
			$$\sigma_{\lambda}|_{I_{\ell}}=\left(\begin{array}{cc}\overline{\varepsilon}_{\ell}^{a_{\lambda}}&\ast\\0&\overline{\varepsilon}_{\ell}^{b_{\lambda}}\end{array}\right)$$
			where we can normalize the exponents to be $0\le a_{\lambda}\le\ell-2$ if $\sigma_{\lambda}|_{I_{\ell}}$ is semisimple, $1\le a_{\lambda}\le\ell-1$ otherwise, and $0\le b_{\lambda}\le\ell-2$. 
		\end{enumerate}
		
		We select $\ell$ sufficiently large so that $\rho_{\lambda}|_{\mathbb{Q}_{\ell}}$ is crystalline and Fontaine-Laffaille (i.e., its Hodge-Tate weights lie in $[0,\ell-2]$). By Fontaine-Laffaille theory \cite[Theorem 5.5]{FL82}, the exponents $a_{\lambda}$ and $b_{\lambda}$ belong to $[0,C]$. Hence we may shrink the infinite set of $\lambda$ on which the $a_{\lambda}$ are all equal and the $b_{\lambda}$ are all equal, and we denote the common values by $a$ and $b$, respectively. The weight formula (see \cite[Section 2]{Da95}) gives
		$$k_{\lambda}=1+(a+m)+(b+m)+(\ell-1)\min\{a+m,b+m\}+(\ell-1)\delta$$
		where $\delta=0$ or 1, and the case $\delta=1$ occurs exactly when $a=b=-m$ or when $\sigma_{\lambda}\otimes\overline{\varepsilon}_{\ell}^m|_{\mathbb{Q}_{\ell}}$ is \emph{tr\'es ramifi\'ee}.
		
		To bound the weights we need the last two terms in the formula to vanish. We take $m$ so that $\min\{a+m,b+m\}=0$. For the last term, since our compatible system is regular, we cannot have $a=b$. Since $\rho_{\lambda}\otimes\varepsilon_{\ell}^m|_{\mathbb{Q}_{\ell}}$ is Fontaine-Laffaille, it follows from \cite[Proposition 2.3.1]{GHLS17} that any of its reduction is peu ramifi\'ee. Then \cite[Remark 2.1.6]{GHLS17} implies that $\sigma_{\lambda}\otimes\overline{\varepsilon}_{\ell}^m|_{\mathbb{Q}_{\ell}}$ is also peu ramifi\'ee, as it is a quotient of two terms in the associated filtration (by the definition of peu ramifi\'ee) of a reduction of $\rho_{\lambda}\otimes\varepsilon_{\ell}^m|_{\mathbb{Q}_{\ell}}$. From \cite[Example 2.1.4(1)]{GHLS17} we conclude that $\sigma_{\lambda}\otimes\overline{\varepsilon}_{\ell}^m|_{\mathbb{Q}_{\ell}}$ is not tr\'es ramifi\'ee. Hence $\delta=0$ in our case, and $k_{\lambda}=1+a+b+2m$ is bounded.
		
		To bound the levels, note that by Serre's conjecture the level $N_{\lambda}$ equals the Artin conductor of $\sigma_{\lambda}\otimes\overline{\varepsilon}_{\ell}^m$ with the possible factors of $\ell$ removed. Let $S$ be the exceptional set of the compatible system $\{\rho_{\lambda}\}$. For each $p\in S$, the image of the wild inertia subgroup $I_p^w$ at $p$ under $\rho_{\lambda}\otimes\varepsilon_{\ell}^m$ is isomorphic to the fixed finite group $\text{WD}_p(I_p^{w})$, where $\text{WD}_p$ is the Weil-Deligne representation appearing in \autoref{Hydrogen}. The formula for the Artin conductor (see \cite[Chapter VI, Corollary 1']{Se79}) together with the finiteness of $S$ then shows that the levels $N_{\lambda}$ are bounded.
	\end{proof}
	
	We use the following result in the proof.
	\begin{prop}\cite[Proposition 5.3.2]{BLGGT14}\label{Neon}
		Given a regular $E$-rational strictly compatible system $\{\rho_{\lambda}\}$ of some number field $K$. There is a Dirichlet density 1 many rational primes $\ell$ such that for any $\lambda|\ell$ and any irreducible subrepresentation $\sigma\subseteq\rho_{\lambda}$, the semisimple reduction $\overline{\sigma}^{\mathrm{ss}}|_{\mathrm{Gal}_{K(\zeta_{\ell})}}$ is irreducible.
	\end{prop}
	
	\section{The proof}\label{section3}
	\subsection{The non-Lie-irreducible case}\label{3.1}\ 
	
	Assume that $\rho_{\lambda_0}$ is not Lie-irreducible. Under regularity condition, \autoref{Beryllium} implies that we can write
	\begin{equation}\label{1}
	\rho_{\lambda_0}=\text{Ind}^{\mathbb{Q}}_K\sigma
	\end{equation}
	for some number field $K\not=\mathbb{Q}$ and some Lie-irreducible representation $\sigma$ of $K$. We show the following result.
	\begin{thm}
		Let $\{\rho_{\lambda}\}$ be a 6-dimensional regular $E$-rational strictly compatible system of $\mathbb{Q}$. Assume $\rho_{\lambda_0}$ is irreducible and has the form (\ref{1}) for $K\not=\mathbb{Q}$.
		\begin{enumerate}
			\item[(i)] If $[K:\mathbb{Q}]=6$ or 3, then $\rho_{\lambda}$ is irreducible for all $\lambda$.
			\item[(ii)] If $[K:\mathbb{Q}]=2$ and $\rho_{\lambda_0}^{\vee}\otimes\chi\cong\rho_{\lambda_0}$ for some character $\chi$, then $\rho_{\lambda}$ is irreducible for all but finitely many $\lambda$.
		\end{enumerate}
	\end{thm}
	
	\subsubsection{$[K:\mathbb{Q}]=6$}\label{3.1.1}
	
	In such case $\sigma$ is an algebraic $\lambda$-adic character. By class field theory after enlarging the coefficients $E$, this character $\sigma$ extends to an $E$-rational compatible system $\{\sigma_{\lambda}\}$. Then $\{\text{Ind}^{\mathbb{Q}}_K\sigma_{\lambda}\}$ also forms an $E$-rational compatible system. By compatibility each $\rho_{\lambda}$ is isomorphic to $\text{Ind}^{\mathbb{Q}}_K\sigma_{\lambda}$. Then the irreducibility of every $\rho_{\lambda}$ follows from Mackey's irreducibility criterion and the regularity condition. In particular, the system $\{\rho_{\lambda}\}$ is induced from a one-dimensional compatible system of $K$.
	\subsubsection{$[K:\mathbb{Q}]=3$}\label{3.1.2}
	
	The restriction of $\rho_{\lambda_0}$ to the Galois closure of $K/\mathbb{Q}$ decomposes as a direct sum of three 2-dimensional Lie-irreducible representations. Since the formal character of $\rho_{\lambda_0}$ coincides with its restriction to any finite-index subgroup, the formal character of $\rho_{\lambda_0}^{\mathrm{der}}$ has the form
	$$\{x,x^{-1},y,y^{-1},z,z^{-1}\}$$
	with the possibility that two or all of $x,y,z$ are equal.
	
	Suppose some $\rho_{\lambda}$ is reducible. By \autoref{Lithium}.(ii) the formal character of its derived subgroup has the form described above. We show that $\rho_{\lambda}$ cannot have 1- or 3-dimensional irreducible components. Since the formal character of $\rho_{\lambda}$ contains no zero weight, there cannot be 1-dimensional components, and the 3-dimensional component (if any) cannot be induced from a character. Hence by \autoref{Beryllium}, the 3-dimensional component must be Lie-irreducible, and its semisimple type would be either $\mathrm{SO}_3$ or $\mathrm{SL}_3$. The former case $\mathrm{SO}_3$ would introduce a zero weight in the formal character, while the latter case $\mathrm{SL}_3$ would force the formal character to contain three weights summing to zero. Neither could happen in our case. Hence the irreducible decomposition of such a $\rho_{\lambda}$ either has the dimensional type $2+2+2$ or has the type $2+4$. Here, for example, "type $2+2+2$" means that $\rho_{\lambda}$ decomposes as
	$$\rho_{\lambda}=W_{\lambda}\oplus W_{\lambda}'\oplus W_{\lambda}''$$
	with $\dim W_{\lambda}=\dim W_{\lambda}'=\dim W_{\lambda}''=2$.
	
	To save ink we say that a representation $\rho$ has dimensional type $\sum a_i\times b_i$ if it admits a decomposition as in (\ref{4}) with $[K:K_i]=a_i$ and the dimension of each Lie-irreducible $\sigma_i$ is $b_i$. \autoref{Argon} rules out the possibility that any irreducible constituent of $\rho_{\lambda}$ is Lie-irreducible. Hence the dimensional type of our $\rho_{\lambda}$ can only be one of the following:
	\begin{enumerate}
		\item[(a)] $2\times1+2\times1+2\times1$
		\item[(b)] $2\times1+4\times1$
		\item[(c)] $2\times1+2\times2$
	\end{enumerate}
	In the first two cases, each constituent is induced from a character. Hence by class field theory, after enlarging the coefficients, this $\rho_{\lambda}$ fits into a compatible system with the same dimensional type. However, by the compatibility of the system $\{\rho_{\lambda}\}$, this would contradict the irreducibility of $\rho_{\lambda_0}$.
	
	Hence for a reducible $\rho_{\lambda}$, its irreducible decomposition must be of the form
	\begin{equation}\label{7} \rho_{\lambda}=\text{Ind}^{\mathbb{Q}}_{F_1}\chi\oplus\text{Ind}^{\mathbb{Q}}_{F_2}\sigma
	\end{equation}
	where $F_1$ and $F_2$ are quadratic number fields, $\chi$ is a character of $F_1$, and $\sigma$ is a 2-dimensional Lie-irreducible representation of $F_2$.
	
	Recall that $d(\rho_{\lambda})$ denotes the Dirichlet density of rational primes that have a split factor in at least one of the fields appearing in the decomposition (\ref{4}) of $\rho_{\lambda}$. For our $\rho_{\lambda_0}$ we have the following.
	\begin{lem}
		$$d(\rho_{\lambda_0})=\left\{\begin{array}{cc}1/3&\text{if}\ K/\mathbb{Q}\ \text{is Galois}\\2/3&\text{if}\ K/\mathbb{Q}\ \text{is non-Galois}\end{array}\right.$$
	\end{lem}
	\begin{proof}
		This is a standard application of Chebotarev's density theorem. We consider only the case where $K/\mathbb{Q}$ is not Galois. Denote by $L$ its Galois closure. Then $\text{Gal}(L/\mathbb{Q})=S_3$, the symmetric group on three elements $\{1,2,3\}$. Consider a rational prime $p$ that is unramified in $K$. It has a split factor in $K$ if and only if either
		\begin{itemize}
			\item $p$ splits completely in $K$, or
			\item $p\mathscr{O}_K=\mathfrak{P}_1\mathfrak{P}_2$ with the interia degrees of $\mathfrak{P}_1$ and $\mathfrak{P}_2$ equal to 1 and 2, respectively.
		\end{itemize}
		
		Equivalently, in $L$ this means that either 
		\begin{itemize}
			\item $p\mathscr{O}_L$ is completely split, or
			\item $p\mathscr{O}_L=\mathfrak{P}_1\mathfrak{P}_2\mathfrak{P}_3$ with each $\mathfrak{P}_i$ having inertia degree 2.
		\end{itemize}
		 Hence this amounts to requiring that $\text{Frob}_p\in S_3$ be either the identity or a transposition. By Chebotarev's density theorem we obtain
		$$d(\rho_{\lambda_0})=\frac{\#\{e,(1,2),(1,3),(2,3)\}}{\#S_3}=\frac{2}{3}$$
	\end{proof}
	
	For a reducible $\rho_{\lambda}$ of the form (\ref{7}), we have the following density formula.
	\begin{lem}\label{Iron}
	$$d(\rho_{\lambda})=\left\{\begin{array}{cc}3/4&F_1\not=F_2\\1/2&F_1=F_2\end{array}\right.$$
	\end{lem}
	\begin{proof}
		
		Since both $F_1$ and $F_2$ are quadratic fields, a rational prime $p$ having a split factor in $F_i$ is equivalent to $p$ splitting completely in $F_i$. Denote by $A_i$ the set of rational primes that split completely in $F_i$ for $i=1,2$. Then
		$$d(\rho_{\lambda})=d(A_1\cup A_2)=d(A_1)+d(A_2)-d(A_1\cap A_2)=1-d(A_1\cap A_2)$$
		where $d(\cdot)$ denotes the Dirichlet density of a set of rational primes. A rational prime $p$ splits completely in both fields $F_1$ and $F_2$ if and only if it splits completely in their compositum $F_1F_2$. Hence by Chebotarev's density theorem we have
		$d(A_1\cap A_2)=d(F_1F_2)=1/[F_1F_2:\mathbb{Q}]$.
	\end{proof}
	The above two lemmas contradict \autoref{Boron}. Hence  $\rho_{\lambda}$ is irreducible for all $\lambda$.
	
	\begin{remark}\label{Phosphorus}
	
	We have shown that each $\rho_{\lambda}$ is irreducible and is induced from some 2-dimensional $\sigma_{\lambda}$ of some cubic field $K_{\lambda}$. In fact, we can take $K_{\lambda}=K$ in each $\rho_{\lambda}$ up to isomorphism. Moreover, if $K/\mathbb{Q}$ is Galois, then the compatible system $\{\rho_{\lambda}\}$ can be written as induced representations of a 2-dimensional Lie-irreducible compatible system of $K$.
	\begin{proof}
		We have $\rho_{\lambda_0}|_L\otimes\chi\cong\rho_{\lambda_0}|_L$ where $\chi$ is the cubic character of $L$ cutting out $K$, and $L$ is defined as follows: If $K/\mathbb{Q}$ is not Galois, let $M$ be its Galois closure and let $L$ be the unique quadratic subfield of $M$; if $K/\mathbb{Q}$ is Galois, set $L=\mathbb{Q}$. Since $\chi$ has finite image, it extends to a compatible system. By compatibility we therefore obtain $\rho_{\lambda}|_L\otimes\chi\cong\rho_{\lambda}|_L$ for all $\lambda$. Moreover each $\rho_{\lambda}|_L$ is irreducible since $L$ and $K_{\lambda}$ are linearly disjoint. The above isomorphism implies that $\rho|_{KL}$ decomposes as a direct sum of three 2-dimensional Lie-irreducible representations. Hence $\rho|_K$ has dimensional type either $2+2+2$ or $4+2$ due to Clifford theory. Choose a 2-dimensional constituent $\sigma$ of $\rho|_K$. Frobenius reciprocity and the irreducibility of $\rho$ then shows that $\rho$ is induced from $\sigma$.
		
		When $K/\mathbb{Q}$ is cubic Galois, $K$ is necessarily totally real. Restricting the system $\{\rho_{\lambda}\}$ to $K$ and apply \autoref{Oxygen}.(i) yields some $\sigma_{\lambda}$ fits into a compatible system. By compatibility we can therefore write $\{\rho_{\lambda}\}$ as induced representations of a 2-dimensional Lie-irreducible strictly compatible system of $K$.
	\end{proof}
	\end{remark}
	\subsubsection{$[K:\mathbb{Q}]=2$}\label{3.1.3}
	
	In this case $\rho_{\lambda_0}\cong\rho_{\lambda_0}\otimes\eta$ where $\eta$ is the nontrivial character of $\mathbb{Z}/2\mathbb{Z}$ cutting out $K$. By compatibility, this isomorphism extends to the whole compatible system $\rho_{\lambda}\cong\rho_{\lambda}\otimes\eta$ for every $\lambda$. Similarly, after enlarging the coefficients $E$, we have a compatible system of characters $\{\chi_{\lambda}\}$ such that $\chi_{\lambda_0}=\chi$ and $\rho_{\lambda}^{\vee}\otimes\chi_{\lambda}\cong\rho_{\lambda}$ for all $\lambda$.
	
	Suppose that $\rho_{\lambda}$ is reducible for infinitely many $\lambda$. Their irreducible decompositions satisfy the following conditions.
	\begin{enumerate}
		\item[(a)] No irreducible constituent of $\rho_{\lambda}$ can be Lie-irreducible due to \autoref{Argon}.
		\item[(b)] Not every irreducible constituent of $\rho_{\lambda}$ is induced from some character, since otherwise class field theory would imply that $\rho_{\lambda}$ fits into a reducible compatible system. This would contradict the irreducibility of $\rho_{\lambda_0}$ due to the compatibility of the systems. 
	\end{enumerate}
	An easy check on all possible of dimensional types of $\rho_{\lambda}$ shows that the only case is:
	$$\rho_{\lambda}=\text{Ind}^{\mathbb{Q}}_{F_1}\sigma_{\lambda}\oplus\text{Ind}^{\mathbb{Q}}_{F_2}\alpha_{\lambda}$$
	where $\sigma_{\lambda}$ is 2-dimensional Lie-irreducible of $F_1$, and $\alpha_{\lambda}$ is a character of $F_2$, and $F_1, F_2$ are quadratic number fields. Since $d(\rho_{\lambda_0})=1/2$, \autoref{Boron} together with \autoref{Iron} forces $F_1=F_2$, denote this common field by $F$. Since $\rho_{\lambda}\otimes\eta\cong\rho_{\lambda}$, the restriction $\rho_{\lambda}|_K$ decomposes as a direct sum of two 3-dimensional $K$-subspaces. If $K\not=F$, then by the Lie-irreuducibility and regularity of $\chi_{\lambda}$ and $\sigma_{\lambda}$, the restriction $\rho_{\lambda}|_K$ still has dimensional type $2+4$, contradicting the previous decomposition. Hence $K=F$. Finally, $K$ must be imaginary quadratic, for otherwise the induced representation $\text{Ind}_K^{\mathbb{Q}}\chi_{\lambda}$ would fail to be regular.
	
	By class field theory, the 2-dimensional consitituent $W_{\lambda}'$ extends to a strictly compatible system. We shall prove the 4-dimensional part $W_{\lambda}$ is potential automorphic for $\lambda$ sufficiently large. Then the compatibility of the system would contradict the irreducibility of $\rho_{\lambda_0}$. 
	
	It suffices to verify the conditions of \autoref{Carbon}. Condition (i) is obvious. Condition (iii) can be obtained from \autoref{Zinc}. 
	
	To check condition (ii), due to dimensional reasons, we have $W_{\lambda}^{\vee}\otimes\chi_{\lambda}\cong W_{\lambda}$ and $W_{\lambda}\otimes\eta\cong W_{\lambda}$. The first isomorphism together with the irreducibility of $W_{\lambda}$ implies that $W_{\lambda}$ is essentially self-dual. Since $K$ is imaginary quadratic, any complex conjugation $c\in\text{Gal}_{\mathbb{Q}}$ lies outside $\text{Gal}_K$. Consequently $\eta(c)=-1$. Thus if $\chi_{\lambda}$ is an odd (resp. even) similitude character then $\chi_{\lambda}\eta$ is an even (resp. odd) similitude character. Hence $W_{\lambda}$ is essentially self-dual and odd.
	
	Finally, we verify condition (iv), namely that for suffciently large $\lambda$, the representation $$(\text{Ind}_{K}^{\mathbb{Q}}\sigma_{\lambda})|_{\mathbb{Q}(\zeta_{\ell})}$$ 
	is residually irreducible, where $\ell$ is the rational prime below $\lambda$. For sufficiently large $\ell$, the field $\mathbb{Q}(\zeta_{\ell})$ and $K$ are linearly disjoint. Hence we have
	$$\text{Res}_{\mathbb{Q}(\zeta_{\ell})}\text{Ind}_{K}^{\mathbb{Q}}\sigma_{\lambda}=\text{Ind}_{K(\zeta_{\ell})}^{\mathbb{Q}(\zeta_{\ell})}\text{Res}_{K(\zeta_{\ell})}\sigma_{\lambda}$$
	As $\sigma_{\lambda}$ is Lie-irreducible, its restriction $\text{Res}_{K(\zeta_{\ell})}\sigma_{\lambda}$ remains irreducible. By regularity and Mackey's irreducibility criterion, the right hand side is therefore also irreducible. Moreover this right hand side is of type A since it is induced from a type A representation. Applying \autoref{Magnesium}.(iv) we conclude that it is residually irreducible for $\lambda\gg0$.

	\begin{remark}
		Let $\{\rho_{\lambda}\}$ be a 3-dimensional regular strictly compatible system of an imaginary quardatic field $K$. Suppose that some $\rho_{\lambda_0}$ is irreducible and some $\rho_{\lambda_1}$ decomposes as a direct sum of a character and a 2-dimensional irreducible $\sigma_{\lambda_1}$. To the author's knowledge, modularity theorems (for $\sigma_{\lambda_1}$) do not seem to be currently available to prove the irreducibility conjecture in this setting. By inducing the system to $\mathbb{Q}$, we obtain a 6-dimensional compatible system that falls into the present case. The extra condition (essential self-duality) we impose is equivalent to requiring that $\rho_{\lambda_0}$ and its conjugate $\rho_{\lambda_0}^c$ (for a complex conjugation $c$) have equal determinant.
	\end{remark}
	\subsection{The Lie-irreducible case}\label{3.2}\ 
	
	Given a Lie-irreducible representation $\rho$, denote by $\rho^{\text{der}}$ its restriction to the derived subgroup, which remains irreducible. In the proof we will refer to the unique $\mathbb{Z}$-model of $\rho^{\text{der}}$ as the semisimple type of $\rho$. For a 6-dimensional Lie-irreducible representation there are exactly 7 semisimple types:
	\begin{enumerate}
		\item[(a)] $(\text{SL}_2,\text{Sym}^5(\text{std}))$
		\item[(b)] $(\text{SL}_2\times\text{SL}_2,\text{std}\otimes\text{Sym}^2(\text{std}))$
		\item[(c)] $(\text{SO}_6,\text{std})$
		\item[(d)] $(\text{Sp}_6,\text{std})$
		\item[(e)] $(\text{SL}_3,\text{Sym}^2(\text{std}))$
		\item[(f)] $(\text{SL}_2\times\text{SL}_3,\text{std}\otimes\text{std})$
		\item[(g)] $(\text{SL}_6,\text{std})$
	\end{enumerate}
	Here $\text{Sym}^r(-)$ denotes the $r$th symmetric power of a representation, and $\text{Std}$ denotes the standard representation of the corresponding group. We prove the following result.
	\begin{prop}\label{Chromium}
		Given a 6-dimensional $E$-rational regular strictly compatible system $\{\rho_{\lambda}:\mathrm{Gal}_{\mathbb{Q}}\to\text{GL}_6(\overline{E_{\lambda}})\}$. Assume that some $\rho_{\lambda_0}$ is Lie-irreducible.
		\begin{enumerate}[(i)]
			\item If $\rho_{\lambda_0}$ has semisimple type (a),(e),(f), or (g), then each $\rho_{\lambda}$ is Lie-irreducible with the same semisimple type.
			\item If $\rho_{\lambda_0}$ has semisimple type (b) or is fully polarisable (i.e., type (c) or (d)), then for all but finitely many $\lambda$ the representation $\rho_{\lambda}$ is Lie-irreducible with the same semisimple type 
		\end{enumerate}
	\end{prop}
	
	The proof of (i) follows essentially the same line as in \cite{Hu23b}. Because the conditions here are slightly different and for the sake of completeness, we outline the argument.
	\subsubsection{Cases (a) and (g)} In case (a) the formal character of $\textbf{G}_{\lambda_0}^{\text{der}}$ is $$\{x^{-5},x^{-3},x^{-1},x,x^3,x^5\}$$ 
	which cannot be decomposed as a disjoint union of two formal characters of lower-dimensional irreducible representations. In case (g) the derived subgroup $\textbf{G}_{\lambda_0}^{\text{der}}=\text{SL}_6$ has maximal rank 5 forces $\textbf{G}_{\lambda}^{\text{der}}=\text{SL}_6$ for all $\lambda$. Thus in both cases all $\rho_{\lambda}$ are irreducible with the same semisimple type.
	
	\subsubsection{Cases (e) and (f)} The formal character of $\textbf{G}_{\lambda_0}^{\text{der}}$ is
	\begin{itemize}
		\item case (e): $\{x^2,xy,y^{-1},y^2,x^{-1},x^{-2}y^{-2}\}$
		\item case (f): $\{xy,xz,xy^{-1}z^{-1},x^{-1}y,x^{-1}z,x^{-1}y^{-1}z^{-1}\}$
	\end{itemize}
	In neither case does the formal character contain a zero weight, nor do the weights occur in pairwise-inverse pairs. Consequently, if some $\rho_{\lambda}$ were reducible, it could not have irreducible constituent of dimension 1 or 2. Hence the only possible irreducible decomposition of $\rho_{\lambda}$ would be of dimensional type $3+3$. But since the formal characters in both cases contain no zero weight, the two 3-dimensional conponents must be Lie-irreducible and the derived subgroups of their algebraic monodromy groups must be $\text{SL}_3$. But this does not match the weights of $\textbf{G}_{\lambda_0}^{\text{der}}$. Hence every $\rho_{\lambda}$ is irreducible with the same semisimple type.
	
	\subsubsection{Cases (c) and (d)}\label{3.2.3} We assume that infinitely many $\rho_{\lambda}$ are reducible. The formal character of $\textbf{G}_{\lambda_0}^{\text{der}}$ in both cases is $\{x,x^{-1},y,y^{-1},z,z^{-1}\}$. For a reducible $\rho_{\lambda}$, the same argument as in the second paragraph of \ref{3.1.2} shows that it contains no irreducible constituents of dimension 1 or 3. Hence $\rho_{\lambda}$ must contain a 2-dimensional irreducible constituent. Moreover, this 2-dimensional component must be Lie-irreducible, otherwise by \autoref{Beryllium} it would be induced from a character, and this introduces zero weights in the formal character.
	
	By \autoref{Oxygen}.(i), for some $\lambda_1$ this 2-dimensional irreducible subrepresentation $\sigma_{\lambda_1}$ would fit into a compatible system $\{\sigma_{\lambda}\}$. Define a new strictly compatible system
	$$\{\rho_{\lambda}\oplus\sigma_{\lambda}\}$$
	By \autoref{Helium}, $\rho_{\lambda_1}\oplus\sigma_{\lambda_1}$ has the same algebraic monodromy group as $\rho_{\lambda_1}$, hence it has semisimple rank 3. However by Goursat's lemma, the derived subgroup of the algebraic monodromy group of $\rho_{\lambda_0}\oplus\sigma_{\lambda_0}$ is either $\text{Sp}_6\times\text{SL}_2$ or $\text{SO}_6\times\text{SL}_2$, which has semisimple rank 4. This contradicts \autoref{Lithium}(ii). 
	
	Next we distinguish between the semisimple type $\text{Sp}_6$ and $\text{SO}_6$. Assume that our compatible system $\{\rho_{\lambda}\}$ has infinitely many terms with semisimple type $\text{SO}_6$ and infinitely many with semisimple type $\text{Sp}_6$. If the similitude character $\chi_{\lambda_0}$ is odd (resp. even), \autoref{Manganese} allows us to choose every similitude character $\chi_{\lambda}$ of $\rho_{\lambda}$ to be odd (resp. even). We then restrict attention to those $\rho_{\lambda}$ that belong to one of the following infinite sets:
	\begin{enumerate}[(i)]
		\item In the odd similitude character case: the set of $\rho_{\lambda}$ with semisimple type $\text{Sp}_6$.
		\item In the even similitude character case: the set of $\rho_{\lambda}$ with semisimple type $\text{SO}_6$.
	\end{enumerate}
	
	We will show that in both cases some $\rho_{\lambda}$ becomes automorphic after restriction to a totally real extension. The conclusion then follows from \cite[Corollary 1.3]{BC11}, which implies that a strictly compatible system of a \emph{totally real field} attached to a regular algebraic, essentially self-dual, cuspidal automorphic representation has the desired property: either every representation in the system factors through $\text{GSp}_n$ or every such representation factors through $\text{GO}_n$.
	
	We verify that for suffiently large $\lambda$, the representation $\rho_{\lambda}$ satisfies the conditions in \autoref{Carbon}. Only condition (iv) requires further explanation. Note that when the semisimple type is $\text{SO}_6$, which is of type A, \autoref{Magnesium}.(iii) implies that $\overline{\rho}_{\lambda}^{\text{ss}}$ is irreducible when restricting to $\mathbb{Q}(\zeta_{\ell})$ for all but finitely many $\lambda$.
	
	So we are in case (i) above. We first prove that $\overline{\rho}_{\lambda}^{\text{ss}}$ is irreducible for all but finitely many $\lambda$ in this case. By \autoref{Magnesium}.(i), the formal bi-character of $\rho_{\lambda}$ coincides with that of $\overline{\rho}_{\lambda}^{\text{ss}}$. If the latter is reducible, it would contain a 2-dimensional irreducible component $\overline{\sigma}_{\lambda}$. Since $\rho_{\lambda}$ is odd, if $\overline{\sigma}_{\lambda}$ is even, then $\overline{\sigma}_{\lambda}\otimes\overline{\chi}_{\lambda}$ is not isomorphic to $\overline{\sigma}_{\lambda}$ and is therefore another 2-dimensional component of $\overline{\rho}_{\lambda}$. But this would force the algebraic envelope of $\overline{\rho}_{\lambda}^{\text{ss}}$ to have maximal rank 2, contradicting the fact that the semisimple rank of $\rho_{\lambda}$ is 3. Hence $\overline{\sigma}_{\lambda}$ is odd. Then by \autoref{Sodium}, we obtain a 2-dimensional strictly compatible system $\{\varphi_{\lambda}\}$ such that for infinitely many $\lambda$ we have $\overline{\varphi}_{\lambda}^{\text{ss}}\cong\overline{\sigma}_{\lambda}$. If the system $\{\varphi_{\lambda}\}$ is not Lie-irreducible, then each $\varphi_{\lambda}$ is induced from a character. Consequently, those infinitely many $\overline{\sigma}_{\lambda}$ is also induced from characters. But this would produce zero weights in the formal character of $\rho_{\lambda_0}^{\mathrm{der}}$, which is not the case. Hence each $\varphi_{\lambda}$ must be Lie-irreducible.

	Define an 8-dimensional strictly compatible system by
	$$\{\psi_{\lambda}:=\rho_{\lambda}\oplus\varphi_{\lambda}\}$$
	On one hand, Goursat's lemma asserts that the derived subgroup $\textbf{G}_{\lambda}^{\text{der}}=\text{Sp}_6\times\text{SL}_2$, which has rank 4. On the other hand, for those $\lambda$ with $\overline{\varphi}_{\lambda}^{\text{ss}}\cong\overline{\sigma}_{\lambda}$, we have $\overline{\psi}_{\lambda}^{\text{ss}}=\overline{\rho}_{\lambda}^{\text{ss}}\oplus\overline{\sigma}_{\lambda}$. The algebraic envelope has semisimple rank 3, contradicting \autoref{Magnesium}.(i).
	
	Next we show that $\overline{\rho}_{\lambda}^{\text{ss}}|_{\text{Gal}_{\mathbb{Q}(\zeta_{\ell})}}$ is irreducible for all but finitely many $\lambda$ in the infinitely family (i). We may twist the original system $\{\rho_{\lambda}\}$ by a system of cyclotomic characters so that one (hence each) algebraic monodromy group in the family (i) becomes $\text{GSp}_6$.
	
	Suppose, to the contrary, that infinitely many of these restrictions are reducible. Since the formal bi-character of the algebraic envelope coincides with that of the algebraic monodromy group, $\underline{G}_{\lambda}$ must be one of the group in the chain
	$$\mathbb{G}_m\left(\text{SL}_2\times\text{SL}_2\times\text{SL}_2\right)\subseteq\mathbb{G}_m\left(\text{SL}_2\times\text{Sp}_4\right)\subseteq\text{GSp}_6$$
	Because $\overline{\rho}_{\lambda}^{\text{ss}}$ is irreducible, the middle group $\mathbb{G}_m\left(\text{SL}_2\times\text{Sp}_4\right)$ is ruled out.
	
	If $\underline{G}_{\lambda}=\text{GSp}_6$, then for sufficiently large $\lambda$ the commutants of $\text{Gal}_{\mathbb{Q}(\zeta_{\ell})}$ and $\underline{G}_{\lambda}^{\text{der}}=\text{Sp}_6$ are equal by \autoref{Magnesium}.(ii). Thus $\overline{\rho}_{\lambda}^{\text{ss}}|_{\text{Gal}_{\mathbb{Q}(\zeta_{\ell})}}$ is irreducible for such $\lambda$.
	
	If $\underline{G}_{\lambda}=\mathbb{G}_m\left(\text{SL}_2\times\text{SL}_2\times\text{SL}_2\right)$, \autoref{Magnesium}.(ii) gives a finite Galois extension $L/\mathbb{Q}$\footnote{By the construction of $\underline{G}_{\lambda}$, the index $|\overline{\rho}_{\lambda}^{\text{ss}}(\text{Gal}_{\mathbb{Q}})/\overline{\rho}_{\lambda}^{\text{ss}}(\text{Gal}_{\mathbb{Q}})\cap\underline{G}_{\lambda}(\overline{\mathbb{F}}_{\ell})|$ is bounded as $\lambda$ varies. Hence there exists a Galois extension $L/\mathbb{Q}$ such that $\overline{\rho}_{\lambda}^{\text{ss}}(\text{Gal}_{L})\subseteq\underline{G}_{\lambda}(\overline{\mathbb{F}}_{\ell})$ for $\lambda\gg0$.} such that the commutants of the image $\overline{\rho}_{\lambda}^{\text{ss}}(\text{Gal}_L)$ and the algebraic envelope $\underline{G}_{\lambda}$ are equal. In particular, $\overline{\rho}_{\lambda}^{\text{ss}}|_{\text{Gal}_L}$ decomposes as a direct sum of three 2-dimensional irreducible representations. By \cite[Lemma 4.3]{CG13}, each $\overline{\rho}_{\lambda}^{\text{ss}}$ is induced from a 2-dimensional irreducible representation of some quadratic field $K_{\lambda}\subseteq L$. Hence infinitely many of the fields $K_{\lambda}$ coincide, and we denote this common field by $K$.
	
	For any rational prime $p$ that split in $K$ and lies outside the exceptional set of the compatible system, the trace of $\rho_{\lambda}(\mathrm{Frob}_p)$ vanishes modulo $\lambda$ for an infinite set of $\lambda$. By compatibility, the trace of $\rho_{\lambda}(\mathrm{Frob}_p)$ is independent of $\lambda$, hence it must be actually zero. Recall that $d(\rho)$ denots the Dirichlet density of rational primes on which the trace of $\rho$ is nonzero. The above argument implies $d(\rho_{\lambda})\le1/2$. But since $\rho_{\lambda_0}$ is Lie-irreducible, we have $d(\rho_{\lambda_0})=1$. This contradicts \autoref{Boron}.
	
	\subsubsection{Case (b)} Suppose that $\rho_{\lambda}$ is reducible for infinitely many $\lambda$. The formal character of $\rho_{\lambda_0}^{\mathrm{der}}$ in this case is $$\{x,x^{-1},xy,x^{-1}y,xy^{-1},x^{-1}y^{-1}\}$$ 
	As before, since this formal character contains no zero weight and no three weights summing to zero, a reducible $\rho_{\lambda}$ cannot have irreducible constituents of dimension 1 or 3. Hence the only possible dimensional types for a reducible $\rho_{\lambda}$ are $2+2+2$ and $2+4$. The former case contradicts the irreducibility of $\rho_{\lambda_0}$ by \autoref{Fluorine}. Thus we may assume that every reducible $\rho_{\lambda}$ decomposes as
	\begin{equation}\label{4}
	\rho_{\lambda}=W_{\lambda}\oplus W_{\lambda}'
	\end{equation}
	with $\dim W_{\lambda}=2$ and $\dim W_{\lambda}'=4$.
	
	Since the semisimple rank of $\rho_{\lambda_0}$ is 2, the derived group $\textbf{G}_{\lambda}^{\mathrm{der}}$ is either $\text{Sp}_4$ or $\text{SL}_2\times\text{SL}_2$. The case $\text{Sp}_4$ is excluded since it has no 2-dimensional irreducible representations. Therefore $\textbf{G}_{\lambda}^{\text{der}}=\text{SL}_2\times\text{SL}_2$. The 2-dimensional part $W_{\lambda}$ fits into a strictly compatible system for sufficiently large $\lambda$ by \autoref{Oxygen}.(i), and the 4-dimensional constituent $W_{\lambda}'$ satisfies
	$$W_{\lambda}^{'\text{der}}=\left(\text{SL}_2\times\text{SL}_2,\text{Std}\otimes\text{Std}\right)$$
	
	We show that after possibly twisting the system $\{\rho_{\lambda}\}$ with a system of cyclotomic characters, we may assume the system is connected, i.e. the algebraic monodromy group $\textbf{G}_{\lambda_0}$ (hence for all $\lambda$, by \autoref{Lithium}.(i)) is connected. Denote by $G$ the algebraic monodromy group of $\rho_{\lambda_0}$ and by $G^{\text{der}}=\mathrm{SL}_2\otimes\mathrm{SO}_3\subseteq\mathrm{GL}_6$ its derived subgroup. Then $G$ is contained in the normalizer $N$ of $G^{\mathrm{der}}$ inside $\mathrm{GL}_6$. Since the Dynkin diagrams of $\text{SL}_2$ and $\text{SO}_3$ admit no nontrivial graph automorphisms, the group $G^{\text{der}}\cong\text{SL}_2\times\text{SO}_3$ has no outer automorphisms. Moreover, the centralizer of $G^{\text{der}}$ in $\text{GL}_6$ is $\mathbb{G}_m$. Hence $N\subseteq\mathbb{G}_mG^{\text{der}}$, and consequently if $G$ is disconnected, it must be of the form $\mu_N G^{\text{der}}$. Twisting the whole system by the system of $\ell$-adic cyclotomic characters then makes the algebraic monodromy group of $\rho_{\lambda_0}$ equal to the connected group $\mathbb{G}_mG^{\text{der}}$.
	
	Thus we may assume $\{\rho_{\lambda}\}$ is connected. In particular, for any reducible $\rho_{\lambda}$ the component $W_{\lambda}'$ factors through the identity component $\text{GO}_4^{\circ}$. By \cite[Corollary 2.2.3, Corollary 3.3.8]{LY16}, we may write infinitely many such $W_{\lambda}'$ as
	\begin{equation}\label{2}
	W_{\lambda}'=\sigma_{\lambda}\otimes\sigma_{\lambda}'
	\end{equation}
	where $\sigma_{\lambda}$ and $\sigma_{\lambda}'$ are both 2-dimensional irreducible representations, unramified outside a finite set of rational primes, and crystalline when restricted to $\mathrm{Gal}_{\mathbb{Q}_{\ell}}$, where $\ell$ denotes the rational prime below $\lambda$.
	
	We show for sufficiently large $\lambda$, both $\sigma_{\lambda}$ and $\sigma_{\lambda}'$ fit into compatible systems. Then both $W_{\lambda}$ and $W_{\lambda}'$ fit into strictly compatible systems for sufficiently large $\lambda$. By the compatibility of the system, this contradicts the irreducibility of $\rho_{\lambda_0}$.
	
	Since our compatible system is connected, both $\sigma_{\lambda}$ and $\sigma_{\lambda}'$ are Lie-irreducible. We verify that the conditions in \autoref{Carbon}\footnote{Alternatively, use \autoref{Nickel} and check only the conditions in \autoref{Nitrogen} for $\lambda\gg0$.} hold for sufficiently large $\lambda$. To check condition (ii) we use \autoref{Nitrogen}.
	\begin{itemize}
		\item Condition \autoref{Carbon}.(i) and \autoref{Nitrogen}.(i) hold by construction.
		\item Condition \autoref{Carbon}.(iii) follows, as before, from \autoref{Zinc}.
		\item Condition \autoref{Nitrogen}.(iii) is obvious.
		\item Condition \autoref{Nitrogen}.(iv) is a consequence of two obvious facts: the Hodge-Tate weights of all $\rho_{\lambda}$ are the same, and the set of differences of Hodge-Tate weights of $\sigma_{\lambda}$ and $\sigma_{\lambda}'$ is contained in the set of differences of Hodge-Tate weights of $\rho_{\lambda}$.
	\end{itemize}
	
	Finally for condition \autoref{Carbon}.(iv) and \autoref{Nitrogen}.(ii), it suffices to verify the latter. For a reducible $\rho_{\lambda}$, its symmetric square contains the component
	$$\text{Sym}^2(\sigma_{\lambda})\otimes\text{Sym}^2(\sigma_{\lambda}')$$
	which is irreducible and of type A. Since the algebraic monodromy group of $\rho_{\lambda}$ is connected, this component is furthermore Lie-irreducible. Applying \autoref{Magnesium}.(iii) to the system $\{\text{Sym}^2(\rho_{\lambda})\}$ shows that for sufficiently large $\lambda$, the above constituent is residually irreducible when restricted to $\mathbb{Q}(\zeta_{\ell})$. In particular, both $\text{Sym}^2(\sigma_{\lambda})$ and $\text{Sym}^2(\sigma_{\lambda}')$ are residually irreducible after restriction to $\mathbb{Q}(\zeta_{\ell})$.
	
	\begin{remark}\label{Potassium}
		In our proof of \autoref{main} for the Lie-irreducible case, all arguments extend directly to a regular strictly compatible system of an arbitrary totally real field $F$, except for case (b). In that case we need to write a Galois representation $\text{Gal}_F\to\text{GO}_4^{\circ}(\overline{E_{\lambda}})$ as a tensor product of two 2-dimensional representations with suitable $\ell$-adic Hodge-theoretic properties, which can be done for $F=\mathbb{Q}$. For the obstruction to such a lifting, see \cite[Corollary 3.2.8]{Pa19}. 
		
		We also remark that in order to adapt our proof to the statement that a 6-dimensional fully polarisable compatible system of a totally real field has the same semisimple type for all but finitely many of its representations, one must add the extra condition that the similitude character of $\rho_{\lambda_0}$ is totally odd or totally even.
	\end{remark}
	\subsection{Proof of \autoref{main2}}\label{3.3}\ 
	
	By \autoref{main} we may assume that every $\rho_{\lambda}$ is reducible. We begin by establishing three lemmas.
	\begin{lem}\label{Calcium}
		For all but finitely many $\lambda$, the irreducible component $\sigma_{\lambda}$ of $\rho_{\lambda}$ that satisfies one of the following conditions fits into a strictly compatible system.
		\begin{itemize}
			\item $\dim\sigma_{\lambda}=1$ or 2,
			\item $\dim\sigma_{\lambda}=3$ and its semisimple type is either trivial or $\text{SO}_3$.
		\end{itemize}
	\end{lem}
	\begin{proof}
		If the semisimple type of $\sigma_{\lambda}$ is trivial, by \autoref{Beryllium} it is induced from a character. The conclusion then follows from class field theory and \autoref{Oxygen}.
	\end{proof}
	
	\begin{lem}\label{Cobalt}
		If there are infinitely many $\lambda$ for which $\rho_{\lambda}$ contains a 4-dimensional irreducible component $\sigma_{\lambda}$, then some of these constituents fit into a strictly compatible system.
	\end{lem}
	\begin{proof}
		
	By regularity each $\sigma_{\lambda}$ is either Lie-irreducible or induced from a Lie-irreducible representation. A 4-dimensional Lie-irreducible representation can have only four possible semisimple types: 
	$$(\text{SL}_2,\text{Sym}^3(\text{Std})), (\text{SO}_4,\text{Std}), (\text{SL}_4,\text{Std}), (\text{Sp}_4,\text{Std})$$ 
	Among these, only the last one is not of type A. Thus we may select infinitely many $\lambda$ satisfying one of the following:
	\begin{enumerate}
		\item[(a)] $\sigma_{\lambda}$ is induced from a character.
		\item[(b)] $\sigma_{\lambda}=\text{Ind}^{\mathbb{Q}}_{K_{\lambda}}\eta_{\lambda}$, where $K_{\lambda}$ is quadratic and $\eta_{\lambda}$ is a 2-dimensional Lie-irreducible representation of $K_{\lambda}$.
		\item[(c)] $\sigma_{\lambda}$ is Lie-irreducible with semisimple type A.
		\item[(d)] $\sigma_{\lambda}$ is Lie-irreducible with semisimple type $\text{Sp}_4$.
	\end{enumerate}
	
	As before, class field theory takes care of case (a). For the remaining three cases we show that for some $\lambda$, the representation $\sigma_{\lambda}$ is potentially automorphic, thus it fits into a strictly compatible system. We verify the conditions in \autoref{Carbon}. Conditions (i) and (iii) are treated exactly as before.
	Since $\rho_{\lambda}$ is essentially self-dual and odd, it has a similitude character $\chi_{\lambda}$. For dimensional reasons $\chi_{\lambda}$ is also a similitude character for $\sigma_{\lambda}$. Hence condition (ii) holds. It remains to check that $\overline{\sigma}_{\lambda}^{\text{ss}}|_{\text{Gal}_{\mathbb{Q}(\zeta_{\ell})}}$ is irreducible, where $\ell$ denotes the rational prime below $\lambda$. 
	
	In case (c) this follows directly from \autoref{Magnesium}.(iii). We therefore restrict our attention to cases (b) and (d). We first show that by shrinking the infinite set $\mathcal{L}$ of places $\lambda$, we may assume $\{\sigma_{\lambda}\}_{\lambda\in\mathcal{L}}$ forms a Serre compatible system. Write $\sigma_{\lambda}'$ for the direct complement of $\sigma_{\lambda}$ in $\rho_{\lambda}$. Suppose first for infinitely many $\lambda$, this 2-dimensional $\sigma_{\lambda}'$ is reducible. Write $\sigma_{\lambda}'=\alpha_{\lambda}\oplus\beta_{\lambda}$ as a direct sum of characters. We apply \autoref{Copper} to the representation $\rho_{\lambda}$. Since $\textbf{G}_{\sigma_{\lambda}}^{\mathrm{der}}$ has no zero weight, and for a character $\chi$ the quotient $\textbf{G}_{\chi}/Z(\textbf{G}_{\chi})$ is trivial, and furthermore under the regularity condition, the character $\alpha_{\lambda}\beta_{\lambda}^{-1}$ cannot have finite image. It follows that the set $\{\xi_a,a\in A\}$ is precisely $\{\alpha_{\lambda},\beta_{\lambda}\}$, and hence the weak abelian part is $\rho_{\lambda}^{\mathrm{wab}}=\alpha_{\lambda}\oplus\beta_{\lambda}$. After enlarging the coefficients, we extend this representation to a 2-dimensional semisimple Serre compatible system $\{\gamma_{\lambda}\}$. Choose another place $\lambda'$ such that $\sigma_{\lambda'}'$ is reducible. Then $\gamma_{\lambda'}$ is a weak abelian direct summand of $\rho_{\lambda'}$ and has the same dimension as $\rho_{\lambda'}^{\mathrm{wab}}=\alpha_{\lambda'}\oplus\beta_{\lambda'}$. Hence necessarily $\gamma_{\lambda'}=\alpha_{\lambda'}\oplus\beta_{\lambda'}$. Consequently, we may take $\mathcal{L}$ to be the set of places such that $\sigma_{\lambda}'$ is reducible.
	
	Now suppose that $\sigma_{\lambda}'$ is irreducible for sufficiently large $\lambda$. By \autoref{Nitrogen} and \autoref{Hydrogen}.(iii) we may further assume that each $\sigma_{\lambda}$ is odd and has fixed Hodge-Tate weights $\{a,b\}$ for $\lambda\gg0$. By twisting the original $\{\rho_{\lambda}\}$ by a suitable power of cyclotomic characters, we may assume $b=0$ and $a>0$. By \autoref{Nickel}, each $\sigma_{\lambda}'$ then is attached to a cuspidal eigenform $f_{\lambda}$ with weight $a+1$. We bound the (minimal) levels $N_{\lambda}$ of $f_{\lambda}$. By \cite[TH\'EOR\'EME (A)]{Ca86} this is equivalent to bounding the Artin conductor of $\sigma_{\lambda}'|_{\mathbb{Q}_{p}}$ for each $p$ in the exceptional set $S$ of $\{\rho_{\lambda}\}$ and $\lambda\nmid p$. But this is bounded by the Artin conductor of the Weil-Deligne representation $\mathrm{WD}_p$ in \autoref{Hydrogen}. Because there can only be finitely many eigenforms with fixed level and weight, we have an infinite set $\mathcal{L}$ of places $\lambda$ such that $f_{\lambda}$ coincide, and $\{\sigma_{\lambda}\}_{\lambda\in\mathcal{L}}$ forms a Serre compatible system.
	
	For case (b), choose one $\lambda_0\in\mathcal{L}$ and let $\chi$ be the character cutting off the quadratic field $K=K_{\lambda_0}$. As $\chi$ has finite image, it extends to a compatible system. From $\sigma_{\lambda_0}\otimes\chi\cong\sigma_{\lambda_0}$ and compatibility, it follows that $\sigma_{\lambda}\otimes\chi\cong\sigma_{\lambda}$ for all $\lambda\in\mathcal{L}$. Thus each $\sigma_{\lambda}$ is induced from a 2-dimensional $\eta'_{\lambda}$ of $K$. This $\eta'_{\lambda}$ is Lie-irreducible because $\eta_{\lambda}$ is. We may further shrink the infinite set $\mathcal{L}$ so that $K$ and $\mathbb{Q}(\zeta_{\ell})$ for $\lambda\in\mathcal{L}$ are linearly disjoint. Then
	\begin{equation}\label{6}
	\text{Res}_{\mathbb{Q}(\zeta_{\ell})}\sigma_{\lambda}=\text{Ind}^{\mathbb{Q}(\zeta_{\ell})}_{K(\zeta_{\ell})}\text{Res}_{K(\zeta_{\ell})}\sigma_{\lambda}
	\end{equation}
	Here $\text{Res}_{K(\zeta_{\ell})}\sigma_{\lambda}$ is Lie-irreducible of type A, so its induced representation is also of type A. Since $\sigma_{\lambda}$ is regular, Mackey's irreducibility criterion shows that the right hand side of (\ref{6}) is irreducible. Applying \autoref{Magnesium}.(iv) we then conclude that $\overline{\sigma}_{\lambda}^{\text{ss}}|_{\mathbb{Q}(\zeta_{\ell})}$ is irreducible for $\lambda\gg0$.
	
	For case (d), choose one $\lambda_0\in\mathcal{L}$ and twist the original system $\{\rho_{\lambda}\}$ by a system of suitable power of cyclotomic characters so that the algebraic monodromy group $\textbf{G}_{\sigma_{\lambda_0}}=\text{GSp}_4$. Then since $\{\sigma_{\lambda}\}_{\lambda\in\mathcal{L}}$ is a Serre compatible system, for each $\lambda\in\mathcal{L}$, the algebraic monodromy group is connected of rank 3 and satisfies
	$$\text{Sp}_4\subseteq\textbf{G}_{\sigma_{\lambda}}\subseteq\text{GSp}_4$$
	hence $\textbf{G}_{\sigma_{\lambda}}=\text{GSp}_4$ for all $\lambda\in\mathcal{L}$. Then the algebraic envelope $\underline{G}_{\sigma_{\lambda}}$ is either $\text{GSp}_4$ or $\mathbb{G}_m\left(\text{SL}_2\times\text{SL}_2\right)$. If infinitely many $\lambda$ fall into the first case, \autoref{Magnesium}.(ii) implies that the commutants of $\overline{\sigma}_{\lambda}^{\text{ss}}(\text{Gal}_{\mathbb{Q}(\zeta_{\ell})})$ and $\text{Sp}_4$ coincide. Thus $\overline{\sigma}_{\lambda}^{\text{ss}}|_{\mathbb{Q}(\zeta_{\ell})}$ is irreducible for $\lambda\gg0$. We may therefore assume $\underline{G}_{\sigma_{\lambda}}=\mathbb{G}_m\left(\text{SL}_2\times\text{SL}_2\right)$ for $\lambda\in\mathcal{L}$. 
	
	We first show that $\overline{\sigma}_{\lambda}^{\text{ss}}$ is irreducible for $\lambda\gg0$. If not, it decomposes as the sum of two 2-dimensional irreducible representations $\alpha_{\lambda}\oplus\beta_{\lambda}$. Since $\sigma_{\lambda}$ admits an odd similitude character $\chi_{\lambda}$, one of the following must hold:
	\begin{itemize}
		\item $\alpha_{\lambda}^{\vee}\otimes\overline{\chi}_{\lambda}\cong\alpha_{\lambda}$ and $\beta_{\lambda}^{\vee}\otimes\overline{\chi}_{\lambda}\cong\beta_{\lambda}$, or
		\item $\alpha_{\lambda}^{\vee}\otimes\overline{\chi}_{\lambda}\cong\beta_{\lambda}$ and $\beta_{\lambda}^{\vee}\otimes\overline{\chi}_{\lambda}\cong\alpha_{\lambda}$
	\end{itemize}
	We rule out the second case since otherwise the semisimple type of $\underline{G}_{\lambda}$ would be $\text{SL}_2$, whose semisimple rank does not match that of $\sigma_{\lambda}$. Thus for sufficiently large $\lambda$, both $\alpha_{\lambda}$ and $\beta_{\lambda}$ are odd.
	
	By \autoref{Sodium} there exist infinitely many $\lambda_1$ and two 2-dimensional strictly compatible systems $\{\widetilde{\alpha}_{\lambda}\}$ and $\{\widetilde{\beta}_{\lambda}\}$ such that the semisimple reductions of $\widetilde{\alpha}_{\lambda_1}$ and $\widetilde{\beta}_{\lambda_1}$ are $\alpha_{\lambda_1}$ and $\beta_{\lambda_1}$, respectively. Write $\rho_{\lambda_1}=\sigma_{\lambda_1}\oplus\sigma'_{\lambda_1}$ and consider the compatible system
	$$\left\{\psi_{\lambda}=\widetilde{\alpha}_{\lambda}\oplus\widetilde{\beta}_{\lambda}\oplus\rho_{\lambda}\right\}$$
	
	For a subrepresentation $\rho\subseteq\psi_{\lambda}$, denote by $G^{\text{der}}(\rho)$ and $\underline{G}^{\text{der}}(\rho)$ respectively the derived subgroups of its algebraic monodromy and of its algebraic envelope. For a semisimple group $G$ write $r(G)$ for its rank. We choose $\lambda_1$ sufficiently large so that \autoref{Magnesium}.(i) applies to this new system $\{\psi_{\lambda}\}$. In particular
	$$r(G^{\text{der}}(\psi_{\lambda_1}))=r(\underline{G}^{\text{der}}(\psi_{\lambda_1}))$$
	From the construction of $\{\widetilde{\alpha}_{\lambda}\}$ and $\{\widetilde{\beta}_{\lambda}\}$, we have
	$$r(\underline{G}^{\text{der}}(\psi_{\lambda_1}))=r(\underline{G}^{\text{der}}(\rho_{\lambda_1}))=r(G^{\text{der}}(\rho_{\lambda_1}))$$
	Combining these equalities we obtain
	\begin{equation}\label{5} r(G^{\text{der}}(\psi_{\lambda_1}))=r(G^{\text{der}}(\rho_{\lambda_1}))
	\end{equation}
	On the other hand, as $r(G^{\text{der}}(\widetilde{\alpha}_{\lambda_1}\oplus\widetilde{\beta}_{\lambda_1}))=r(\underline{G}^{\text{der}}(\widetilde{\alpha}_{\lambda_1}\oplus\widetilde{\beta}_{\lambda_1}))=r(G^{\text{der}}(\sigma_{\lambda_1}))=2$, it follows that
	$$r(G^{\text{der}}(\widetilde{\alpha}_{\lambda_1}\oplus\widetilde{\beta}_{\lambda_1}\oplus\sigma'_{\lambda_1}))\ge2>1\ge r(G^{\text{der}}(\sigma'_{\lambda}))$$
	The representation $\widetilde{\alpha}_{\lambda_1}\oplus\widetilde{\beta}_{\lambda_1}\oplus\sigma'_{\lambda_1}$ is of type A, the semisimple type of $\sigma'_{\lambda_1}$ is either type A or trivial. Moreover $\sigma_{\lambda_1}$ is of type $\text{Sp}_4$. By Goursat's lemma,
	$$G^{\text{der}}(\psi_{\lambda_1})=G^{\text{der}}(\widetilde{\alpha}_{\lambda_1}\oplus\widetilde{\beta}_{\lambda_1}\oplus\sigma'_{\lambda_1})\times\text{Sp}_4$$
	$$G^{\text{der}}(\rho_{\lambda_1})=G^{\text{der}}(\sigma'_{\lambda_1})\times\text{Sp}_4$$
	Therefore
	$$r(G^{\text{der}}(\psi_{\lambda_1}))>r(G^{\text{der}}(\rho_{\lambda_1}))$$
	which contradicts (\ref{5}).
	
	Now we assume that $\overline{\rho}_{\lambda}^{\text{ss}}$ is irreducible but its restriction $\overline{\rho}_{\lambda}^{\text{ss}}|_{\mathbb{Q}(\zeta_{\ell})}$ is reducible for $\lambda\gg0$. By \autoref{Magnesium}.(ii), there is a Galois extension $L/\mathbb{Q}$ such that the commutants of $\overline{\sigma}_{\lambda}^{\text{ss}}(\text{Gal}_L)$ and $\underline{G}_{\sigma_{\lambda}}$ coincide. Hence each $\overline{\sigma}_{\lambda}^{\text{ss}}$ is induced from a 2-dimensional irreducible representation $\theta_{\lambda}$ of some quadratic field $E_{\lambda}\subseteq L$. After shrinking the infinite set of $\lambda$ we may assume all $E_{\lambda}$ are equal and we denote this common field by $E$. Shrink further we can arrange that $E$ and $\mathbb{Q}(\zeta_{\ell})$ are linearly disjoint as $\ell=\ell(\lambda)$ varies. Then
	$$\overline{\sigma}_{\lambda}^{\text{ss}}|_{\mathbb{Q}(\zeta_{\ell})}=\text{Ind}^{\mathbb{Q}(\zeta_{\ell})}_{E(\zeta_{\ell})}\left(\theta_{\lambda}|_{E(\zeta_{\ell})}\right)$$
	We apply Mackey's irreducibility criterion, which requires:
	\begin{enumerate}
		\item[(i)] The 2-dimensional $\theta_{\lambda}|_{E(\zeta_{\ell})}$ is irreducible.
		\item[(ii)] For any $s\in\text{Gal}_{\mathbb{Q}(\zeta_{\ell})}\backslash\text{Gal}_{E(\eta_{\ell})}$ denote by $\theta^s_{\lambda}$ the representation $g\mapsto\theta_{\lambda}(s^{-1}gs)$ of $E(\zeta_{\ell})$. Then $\theta_{\lambda}$ and $\theta^s_{\lambda}$ are not isomorphic.
	\end{enumerate}
	For condition (i), if it splits into a sum of two characters, the algabraic envelope of $\sigma_{\lambda}$ would be trivial, contradicting \autoref{Magnesium}.(i). For condition (ii), choose $\lambda$ sufficiently large and a place $v$ of $E(\zeta_{\ell})$ outside the exceptional set, so that the local representation $\rho=\rho_{\lambda}|_{E(\zeta_{\ell})_v}$ is crystalline, regular and Fontaine-Laffaille (i.e. the absolute differences of its Hodge-Tate weights lie in $[1,\ell-2]$). Fontaine-Laffaille theory \cite[Theorem 5.5]{FL82} (see also \cite[Theorem 1.0.1]{Ba20}) together with regularity then asserts that each irreducible component of $\overline{\rho}^{\text{ss}}$ has multiplicity one. In particular $\theta_{\lambda}$ and $\theta^s_{\lambda}$ cannot be isomorphic.
	\end{proof}
	
	\begin{lem}\label{Titanium}
		If the purity condition holds, then no $\rho_{\lambda}$ can contain a 3-dimensional irreducible constituent $\sigma_{\lambda}$ whose semisimple type is $\text{SL}_3$.
	\end{lem}
	\begin{proof}
		Suppose, for contradiction, that such component exists for $\lambda_0$. By essential self-duality, $\rho_{\lambda_0}$ must decompose as  $\rho_{\lambda_0}=\sigma_{\lambda_0}\oplus \sigma'_{\lambda_0}$ with $\sigma'_{\lambda_0}=\sigma^{\vee}_{\lambda_0}\otimes\chi_{\lambda_0}$. Hence there cannot be infinitely many $\lambda$ for which every irreducible constituent of $\rho_{\lambda}$ belongs to one of the types listed in \autoref{Calcium} and \autoref{Cobalt}. Indeed, otherwise those lemmas would imply the existence of either
		\begin{itemize}
			\item a 3-dimensional strictly compatible system $\{\psi_{\lambda}\}$ such that $\psi_{\lambda_0}=\sigma_{\lambda_0}$ while another $\psi_{\lambda_1}$ is either irreducible of semisimple type $\text{SO}_3$ or trivial, or has a character as a constituent; or
			\item a 4-dimensional strictly compatible system $\{\psi_{\lambda}\}$ such that for some $\psi_{\lambda_1}$ is irreducible while $\psi_{\lambda_0}$ contains $\sigma_{\lambda_0}$ as a constituent.
		\end{itemize}
		In either case, the formal characters of $\psi_{\lambda_0}^{\text{der}}$ and $\psi_{\lambda_1}^{\text{der}}$ do not match.
		
		Thus for all but finitely many $\lambda$, the only possible dimensional type of $\rho_{\lambda}$ other than $3+3$ is $5+1$. Denote by $d$ the Dirichlet density of those $\lambda$ for which $\rho_{\lambda}$ has dimensional type $5+1$. If $d>0$, then \autoref{Neon} implies that infinitely many of those 5-dimensional components satisfy \autoref{Carbon}.(iv). Since by assumption $\rho_{\lambda}$ is essentially self-dual and odd, and admits a 5-dimensional irreducible subrepresentation, $\rho_{\lambda}$ is necessarily of orthogonal type with an even similitude character $\chi_{\lambda}$. Then for dimensional reasons $\sigma_{\lambda}\cong \sigma_{\lambda}^{\vee}\otimes\chi_{\lambda}$, thus $\sigma_{\lambda}$ satisfies \autoref{Carbon}.(ii). The remaining conditions of \autoref{Carbon} are easily verified. Hence we could write $\{\rho_{\lambda}\}$ as a direct sum of a 5-dimensional compatible system and a 1-dimensional compatible system. This contradicts the decomposition of $\rho_{\lambda_0}$. Hence $d$ must be zero. 
		
		Thus we may assume there is a set of Dirichlet density one of primes $\lambda$ such that $\rho_{\lambda}=\sigma_{\lambda}\oplus\sigma_{\lambda}'$ where $\sigma'_{\lambda}=\sigma^{\vee}_{\lambda}\otimes\chi_{\lambda}$. On the other hand, \cite[Lemma1.5, Lemma 1.6]{PT15}\footnote{Lemma 1.6 applies to compatible systems of totally real fields as well.} assert that there exists a set of primes $\lambda$ of positive Dirichlet density such that for any irreducible constituent $\sigma\subseteq\rho_{\lambda}$, the Hodge-Tate weights of $\sigma$ and $\sigma^{\vee}\otimes\chi_{\lambda}$ are the same. This contradicts the regularity condition.
	\end{proof}
	
	We now prove \autoref{main2}. Let $r$ be the smallest integer for which there exist infinitely many $\lambda$ such that each irreducible component of $\rho_{\lambda}$ has dimension at most $r$. The three lemmas above assert that our result is true for $r\le4$. Assume $r=5$. Then for sufficiently large $\lambda$, the representation $\rho_{\lambda}$ has dimensional type $5+1$. The same argument outlined in the second paragraph of \autoref{Titanium} then shows that some 5-dimensional constituent fits into a compatible system. This completes the proof.
	
	\section*{Acknowledgement}
	I would like to express my deepest gratitude to my advisor, Chun Yin Hui, for his invaluable guidance, insightful suggestions, unwavering support, and encouragement throughout not only this research, but also my Ph.D. journey. I am deeply appreciative of his time, effort, and dedication. His mentorship has been a cornerstone of my academic journey.


\end{document}